\documentclass[12pt]{amsart}

\usepackage[centertags]{amsmath}
\usepackage{amsfonts}
\usepackage{amssymb}
\usepackage{amsthm}
\usepackage[english]{babel}
\usepackage[active]{srcltx}

\usepackage{pst-all}

\newcommand{\R}{\mathbb{R}}
\newcommand{\B}{B\,}

\newcommand{\co}{co}

\newtheorem{thm}{Theorem}[section]
\newtheorem{cor}[thm]{Corollary}

\newtheorem{lema}[thm]{Lemma}
\newtheorem{prop}[thm]{Proposition}

\newtheorem{rem}[thm]{Remark}
\newtheorem{ex}[thm]{Example}

\newtheorem{defn}[thm]{Definition}

\numberwithin{equation}{section}

\begin{document}

\title[Inversion of Nonsmooth Maps]{Inversion of Nonsmooth Maps between Banach spaces}

\author{Jes\'us A. Jaramillo$^1$, Sebasti\'an Lajara$^2$ and \'Oscar  Madiedo$^3$}

\address{$^1$ Instituto de Matem\'atica Interdisciplinar (IMI) and Departamento de An{\'a}lisis Matem{\'a}tico\\  Universidad Complutense de Madrid\\ 28040 Madrid, Spain}

\email{jaramil@mat.ucm.es}

\address{$^2$ Departamento de Matem\'aticas\\ Universidad de Castilla la Mancha\\ Escuela de Ingenieros Industriales, 02071 Albacete, Spain}

\email{sebastian.lajara@uclm.es}

\address{$^3$ Departamento de Matem\'atica Aplicada, Ciencia e Ingenier\'ia de los Materiales y Tecnolog\'ia Electr\'onica\\ Universidad Rey Juan Carlos\\  Calle Tulip\'an, s/n, 28933 M\'ostoles,  Spain}

\email{oscar.madiedo@urjc.es}

\thanks{Research supported in part by MICINN under Project MTM2015-65825-P (Spain). The research of the second author is also supported by MICINN under Project MTM2014-54182-P (Spain) and by Fundaci\'on S\'eneca (Agencia de Ciencia y Tecnolog\'ia de la Regi\'on de Murcia) under Project 19275/PI/14. The research of the third author is also supported by IMI
(Instituto de Matem\'atica Interdisciplinar, Universidad Complutense de Madrid), REDIum (Red de Institutos Universitarios de Matem\'aticas de Espa\~{n}a) and IMACI (Instituto de Matem\'atica Aplicada a la Ciencia y la Ingenier\'ia, Universidad de Castilla-La Mancha).}

%%Different parts of this paper were developed during a visit of \'{O}. Madiedo to IMACI (Albacete).}



\date{\today}                                           % Activate to display a given date or no date

%%%%%%%%%%%%%%%%%%%%%%%%%%%%%%%%%%%%%%%%%%
%%%%%%%%%%%%%%%%%%%%%%%%%%%%%%%%%%%%%%%%%%%
%%%%\maketitle

\begin{abstract}
We study the invertibility nonsmooth maps between infinite-dimensional Banach spaces. To this end, we introduce
an analogue of the notion of pseudo-Jacobian matrix of Jeyakumar and Luc in this infinite-dimensional setting.
Using this, we obtain several inversion results. In particular, we give a version of the classical Hadamard integral condition for global invertibility in this context.

\end{abstract}
%%%%%%%%%%%%%%%%%%%%%%%%%%%%%%%%%%%%%%%%%%%%
%%%%%%%%%%%%%%%%%%%%%%%%%%%%%%%%%%%%%%%%%%%%

\maketitle

\section{Introduction}

Invertibility of maps is a fundamental issue in nonlinear analysis. In the smooth setting, if $f:X\to Y$ is a $C^1$-map between Banach spaces, such that its derivative $f'(x)$ is an isomorphism for every $x\in X$, of course from the classical Inverse Function Theorem we have that $f$ is locally invertible around each point. If $f$ satisfies in addition the so-called {\it Hadamard integral condition,} that is, if
$$
\int_0^{\infty} \inf_{\Vert x \Vert \leq t} \Vert f'(x)^{-1} \Vert^{-1} \, {\rm d}t = \infty,
$$
then $f:X\to Y$ is globally invertible, and thus a global diffeomorphism from $X$ onto $Y$. We refer to Plastock \cite{Plastock} for a proof of this result. Note that, in particular, the Hadamard integral condition is satisfied whenever
$$
\sup_{x\in X} \Vert f'(x)^{-1} \Vert < \infty,
$$
a requirement which is often referred to as the {\it Hadamard-Levy condition.} These kinds of sufficient conditions for global invertibility were first considered by Hadamard \cite{Hadamard} for maps between finite-dimensional spaces, and then by L\'evy \cite{Le} in the case of maps between Hilbert spaces, and have been widely used since then. We refer to the recent survey of Gut\'u \cite{Gu} for an extensive and detailed information about these and other conditions for global inversion of smooth maps between Banach spaces and, more generally, between Finsler manifolds.

\

In a nonsmooth setting, if $f: \mathbb R^n \to \mathbb R^n$ is a locally Lipschitz map, Clarke obtained in \cite{Clarke} a local inversion result using the {\it Clarke generalized Jacobian} $\partial f$. The corresponding global inversion theorem, using a suitable version of the Hadamard integral condition in this setting, was obtained by Pourciau in \cite{Pourciau1982} and \cite{Pourciau1988}, where it was shown that $f$ is globally invertible provided
$$
\int_0^{\infty} \inf_{\Vert x \Vert \leq t} \inf \{ \Vert T^{-1} \Vert^{-1} :  T \in \partial f (x)\}  \, {\rm d}t = \infty.
$$
An analogous result was achieved in \cite{JaMaSa} for locally Lipschitz maps between finite-dimensional Finsler manifolds. For continuous maps  $f: \mathbb R^n \to \mathbb R^n$ which are not assumed to be locally Lipschitz,  Jeyakumar and Luc introduced in \cite{JL0} the concept of {\it approximate Jacobian matrix,} which later on was named {\it pseudo-Jacobian matrix} (see \cite{JL}). Local inversion results using pseudo-Jacobian matrices have been obtained in \cite{JL1} and \cite{Luc}. Furthermore, a global inversion theorem in terms of such matrices is given in \cite{JaMaSa1},  with a version of the Hadamard integral condition in this context.

\

If $f:X\to Y$ is a nonsmooth map between infinite-dimensional Banach spaces, the problem of local invertibility of $f$ is more delicate. Assuming that $f$ is a local homeomorphism, F. John obtained in \cite{John} a global inversion theorem using a suitable version of the Hadamard integral condition in terms of the {\it lower scalar  Dini derivative} of $f$, which is defined for  each $x\in X$ as
$$
D^{-}_x f =\liminf_{z\to x} \frac{\|f(z) - f(x)\|}{\|z - x\|}.
$$
Namely, he proved that $f$ is globally invertible if
$$
\int_0^{\infty} \inf_{\Vert x \Vert \leq t}  D^{-}_x f \, {\rm d}t = \infty.
$$
Further results along this line have been obtained recently in \cite{GutuJaramillo} and \cite{GaGuJa}  in the more general setting of maps between metric spaces.

\

Our aim here is to study the invertibility of continuous, nonsmooth maps between infinite-dimensional Banach spaces, using a suitable notion of pseudo-Jacobian. The contents of the paper are as follows. In Section 2 we define pseudo-Jacobians in this setting  and we obtain some of their basic properties. Given a map $f:X \to Y$ between Banach spaces and a point $x\in X$, a {\it pseudo-Jacobian} for $f$ at $x$ will be a subset of continuous linear operators $Jf(x) \subset \mathcal{L}(X, Y)$ which provides some control of the size of the upper Dini directional derivatives at $x$ of all compositions of the form $y^*\circ f$, where $y^*$ runs over all continuous linear functionals on $Y$ (see Definition \ref{pseudo-jacobian}). We first show a variety of examples of pseudo-Jacobians, such as the classical weak-G\^ateaux derivative, or the G\^ateaux prederivative in the sense of Ioffe (see \cite{Ioffe}). For locally Lipschitz real-valued functions, the Clarke subdifferential is a pseudo-Jacobian; more generally, in the case that the map is locally Lipschitz and the target space is {\it reflexive}, we see that the generalized Jacobian of P\'ales and Zeidan (see \cite{PZ-07}),  which is an infinite-dimensional extension of the Clarke generalized Jacobian, constitutes an example of pseudo-Jacobian. Next, we obtain a suitable version of the Mean Value inequality using pseudo-Jacobians in Theorem \ref{MVT}, and an optimality condition in Theorem \ref{optimality}. Since the chain rule does not hold for G\^ateaux differentiability, of course we cannot expect a general exact chain rule in our setting, but we obtain some partial results in this direction which will be useful along the paper. In particular, we give in Proposition \ref{chainrulebis} and Theorem \ref{chainrule}, respectively, a smooth-nonsmooth and a nonsmooth-smooth version of the chain rule, when one of the involved maps has some differentiability properties. Section 3 is devoted to study the invertibility of nonsmooth maps. Our first result (Theorem \ref{main-open}) is an open mapping theorem involving the Rabier surjectivity index (see \cite{Rabier}) for pseudo-Jacobians satisfying a technical condition, which we call {\it chain rule condition} (see Definition \ref{crc}). This condition applies, in particular, to the case of G\^ateaux derivatives, and thus we obtain as a special case the Ekeland's open mapping theorem in \cite{Ekeland} (see Corollary \ref{Ekeland}). Subsequently, we establish a local inversion theorem in terms of a new index for pseudo-Jacobians satisfying the chain rule condition (see Theorem \ref{local-inversion}).
This result is used to achieve a global inversion theorem in terms of a suitable version of the Hadamard integral condition in this setting (see Theorem \ref{main-H}). In the case of maps between reflexive spaces, we deduce in Corollary \ref{pour} a global inversion theorem involving upper semicontinuous {\it P\'ales-Zeidan generalized Jacobians} (see \cite{PZ-07}). In this way we extend the results of Pourciau in \cite{Pourciau1982} and \cite{Pourciau1988} to the infinite-dimensional setting.  The aforementioned results are applied to obtain several sufficient conditions regarding global inversion of Lipschitz perturbations of smooth maps. Finally, we combine our local inversion theorem together with Plastock's limiting condition of global invertibility of local homeomorphisms in \cite{Plastock} to give a full characterization of homeomorphisms between Banach spaces in terms of pseudo-Jacobians (see Theorem \ref{main2}).

\section{Pseudo-Jacobians}

In this section, we study the analogue of pseudo-Jacobian matrices of Jeyakumar and Luc (see \cite{JL}) in the setting of arbitrary Banach spaces. First, we fix some notations.
In what follows, $X$ and $Y$ will denote Banach spaces and $U$ a (nonempty) open subset of $X$. The symbols $X^*$, $B_X$ and $\overline{B}_X$ will stand for the topological dual of $X$ and the open and closed unit balls of $X$, respectively, and the space of bounded linear operators from $X$ into $Y$ will be denoted as $\mathcal{L}(X,Y)$. It is natural to consider on this space the \emph{weak operator topology} ({\tiny \emph{WOT}} for short),  that is, the topology of the pointwise convergence on $X$ with respect to the weak topology on $Y$; this means that a net $(T_i)_i$ is {\tiny \emph{WOT}}-convergent to $T$  in $\mathcal{L}(X,Y)$ if, and only if, for each $y^*\in Y$ and each $v\in X$ the net $(\left\langle y^*,\, (T_i-T)(v)\right\rangle)_i$ converges to zero. We shall also deal along the paper with the \emph{strong operator topology} ({\tiny \emph{SOT}} for short), that is, the topology of the pointwise convergence on $X$ with respect to the norm topology on $Y$. Finally, we recall that if $\varphi :U \to \mathbb R$ is a real-valued function and  $x$ is a point in $U$, then the \emph{upper} and \emph{lower right-hand Dini derivatives} of $\varphi$ at  $x$ with respect to a vector $v\in X$ are defined as
$$
\varphi'_{+}(x;\, v) = \limsup_{t\to 0^+} \frac{\varphi(x+tv)-\varphi(x)}{t}
\quad \quad \text{and
} \quad \quad
\varphi'_{-}(x;\, v) = \liminf_{t\to 0^+} \frac{\varphi(x+tv)-\varphi(x)}{t}.
$$
We refer to \cite{Clarkebook}, \cite{Fabian} or \cite{Ph} for the definition and basic properties of different kinds of differentiability of maps between Banach spaces and other unexplained notions.

\

The following is the main concept of this work.

\begin{defn}\label{pseudo-jacobian}
Let $X$ and $Y$ be Banach spaces,  $U$ be an open subset of $X$ and $f: U \to Y$  a map. We say that a nonempty subset  $Jf(x)\subset \mathcal{L}(X, Y)$ is a pseudo-Jacobian of $f$ at a point $x\in U$ if
\begin{equation}
(y^* \circ f)'_{+}(x;\, v) \leq \sup \{ \langle y^*,\, T(v) \rangle \, : \, T\in Jf(x) \} \quad \text{whenever} \,\, y^* \in Y^* \,\, \text{and} \,\,  v\in X.\label{pd-1}
\end{equation}
A set-valued mapping $Jf:U\to 2^{\mathcal{L}(X,Y)}$ is called a pseudo-Jacobian mapping for $f$ on $U$ if for every $x\in U$, the set $Jf(x)$ is a pseudo-Jacobian of $f$ at $x$.
\end{defn}

Let us summarize some basic facts on pseudo-Jacobians. It is clear that, if $Jf(x)$ is a pseudo-Jacobian of $f$ at the point $x$, then any subset of $\mathcal{L}(X,Y)$ containing $Jf(x)$ is also a pseudo-Jacobian of $f$ at $x$. Moreover, for any set $\mathcal{F}\subset \mathcal{L}(X, Y)$ we have
$$
\sup \left\{\left\langle y^*,\, T(v)\right\rangle \, : \, T\in \mathcal{F} \right\}= \sup \left\{\left\langle y^*,\, T(v)\right\rangle \, :\, T\in \overline{\co}^{WOT} (\mathcal{F})\right\},
$$
where $\overline{\co}^{WOT} (\mathcal{F})$ denotes the {\tiny \emph{WOT}}-closed convex hull of $\mathcal{F}$. Thus, a subset $Jf(x)\subset \mathcal{L}(X, Y)$ is a pseudo-Jacobian of $f$ at $x$ if, and only if, so is its {\tiny \emph{WOT}}-closed convex hull.

\

Observe also that, since
$$
(-y^*\circ f)'_{+}(x;\, v)=-(y^*\circ f)'_{-}(x;\, v) \quad \text{ and } \sup \limits_{T\in Jf(x)} \left\langle -y^*,\,T(v) \right\rangle= -\inf \limits_{T\in Jf(x)} \left\langle y^*,\,T(v)\right\rangle,
$$
we have that $Jf(x)$ is a pseudo-Jacobian of $f$ at $x$ if, and only if,
\begin{equation}
(y^*\circ f)'_{-}(x;\, v)\geq \inf\{ \langle y^*,\, T(v) \rangle \, : \, T\in Jf(x) \}, \quad \text{whenever} \, \, y^*\in Y^*\,\, \text{and} \,\, v\in X.\label{pd-2}
\end{equation}
As a consequence of this inequality and \eqref{pd-1} (and the fact that there are only two directions in $\mathbb R$), we deduce that if $\varphi: U \to \mathbb R$ is a real-valued function, then a nonempty subset $J\varphi(x) \subset X^*$ is a pseudo-Jacobian of $f$ at $x\in U$ if, and only if, for every $v\in X$,
\begin{equation}
\inf \limits_{x^* \in J\varphi(x)} x^*(v)\leq \varphi'_{-}(x;\, v)\leq \varphi'_{+}(x;\, v)\leq \sup \limits_{x^* \in J\varphi(x)} x^*(v).\label{pd-scalar}
\end{equation}
Finally, we notice that the same argument as in the finite-dimensional case (see \cite[Theorem 2.1.1]{JL}) yields that if $Jf(x)$ and $Jg(x)$ are pseudo-Jacobians of maps $f,\, g:U\subset X\to Y$ at a point $x\in U$ and $\alpha\in \mathbb R$, then the set $\alpha Jf(x) + Jg(x) = \{\alpha T+S:\, T\in Jf(x),\, S\in Jg(x)\}$ is a pseudo-Jacobian of the map $\alpha f + g$ at this point.

Many examples of pseudo-Jacobians for maps between finite dimensional Banach spaces are given in \cite{JL}.
Next, we will present some examples in the general setting.

\begin{ex}\label{wG}
\rm{
A map $f:U\subset X\to Y$ is said to be \emph{weakly G\^ateaux differentiable} at a point $x\in U$ if
there exists an operator $f'(x)\in \mathcal{L}(X,Y)$ (called the \emph{weak derivative of $f$ at $x$}) such that
$$
\lim_{t\to 0} \left \langle y^*, \frac{f(x+tv)-f(x)}{t}\right \rangle = \left\langle y^*,\, f'(x)(v)\right\rangle, \quad \text{for all} \quad v\in X \quad \text{and} \quad y^*\in Y^*.
$$
It is clear that if $f$ is  weakly G\^ateaux differentiable at $x$, then the set $Jf(x):=\{f'(x)\}$ is a pseudo-Jacobian of $f$ at $x$. Conversely, if $f$ admits a pseudo-Jacobian at $x$ which is a singleton, $Jf(x)=\{T\}$, then $f$ is weakly G\^ateux differentiable at $x$ and $f'(x)=T$. Indeed, according to \eqref{pd-1} and \eqref{pd-2}, for each $y^*\in Y^*$ and each $v\in X$, we get
$$
\left\langle y^*,\, (Tv)\right\rangle \leq (y^* \circ f)_{-}'(x;v)\leq (y^* \circ f)_{+}'(x;v)\leq  \left\langle y^*,\, T(v)\right\rangle,$$
and hence we have the right-hand limit
$$
\lim_{t\to 0^{+}} \left \langle y^*, \frac{f(x+tv)-f(x)}{t}\right \rangle = \left\langle y^*, T(v)\right\rangle.
$$
Furthermore,
$$
\lim_{t\to 0^{-}} \left \langle y^*, \frac{f(x+tv)-f(x)}{t}\right \rangle = \lim_{s\to 0^{+}} \left \langle y^*, \frac{f(x-sv)-f(x)}{-s}\right \rangle = -\left\langle y^*, T(-v)\right\rangle = \left\langle y^*, T(v)\right\rangle.
$$
Thus, $f$ is weakly G\^ateaux differentiable at $x$, with  $f'(x)=T$.} \hfill $ \square$
\end{ex}

\

\begin{ex} \label{Ioffe}
\rm{
Let $f:U\subset X\to Y$ be a map. According to Ioffe \cite{Ioffe}, a nonempty subset $Jf(x) \subset \mathcal{L}(X, Y)$ is said to be a \emph{G\^ateaux prederivative} of $f$ at a point $x\in U$ if, for each $\varepsilon>0$  and each $v\in X$, there exists some $\delta>0$ such that, whenever $\vert t \vert < \delta$,  we have
$$
f(x+tv) - f(x) \in Jf(x)(t v) + \varepsilon \vert t \vert \cdot \overline{B}_{Y},
$$
where
$$
Jf(x)(t v) = \{T(t v) \, : \, T\in Jf(x) \}.
$$
We claim that, if $Jf(x)$ is a  G\^ateaux prederivative of $f$ at the point $x\in U$, then it is a pseudo-Jacobian of $f$ at $x$. Indeed, fix a nonzero vector $v\in X$. Given $\varepsilon >0$, choose $\delta>0$ as in the previous definition. For each $0< t< \delta$ there exist some $T_t\in Jf(x)$ and some $y_t\in \overline{B}_{Y}$ such that,  for all $y^*\in Y^*$,
$$
\frac{1}{t} \big( (y^* \circ f)(x + tv)-(y^* \circ f)(x) \big) = \langle y^*, T_t(v)\rangle + \varepsilon  \langle y^*, y_t\rangle.
$$
It follows that
$$
(y^* \circ f)'_{+}(x;v) \leq \sup\{\langle y^*, T(v) \rangle: T\in Jf(x)\}.
$$
$\hfill  \square$
}
\end{ex}

\

Now, we shall give some examples of pseudo-Jacobians for locally Lipschitz maps. Recall that a map $f:U\subset X\to Y$ is said to be \emph{locally Lipschitz} at a point $x\in U$ if there exist $L,\, r> 0$ such that $\|f(u)-f(v)\|\leq L\|u-v\|$ whenever $u,v$ belong to the open ball $B(x,r)\subset U$. In such case, we define the \emph{local Lipschitz constant} of $f$ at $x$ as the number
$$
{\rm Lip}\, f (x) = \inf_{r>0} \sup \left\{ \frac{\Vert f(u) - f(v) \Vert}{\Vert u -v \Vert}: \,  \,\,\, u, v\in B(x, r)\,\,\, \text{and} \,\,\, u\neq v \right\}.
$$

\

\begin{ex}\label{Clarke}
\rm{
Consider a real-valued function $\varphi:U\subset X\to \mathbb R$  which is locally Lipschitz at a point $x\in U$. Recall that the \emph{Clarke subdifferential} of $\varphi$ at $x$ is defined as
$$
\partial \varphi(x) := \{x^* \in X^* \, : \, x^*(v) \leq \varphi^{\circ}(x;\, v), \,\,\, \text{for all} \,\,\, v \in X\},
$$
where $\varphi^{\circ}(x;\, v)$ is the \emph{Clarke generalized directional derivative} of $\varphi$ at $x$ with respect to a vector $v$, that is,
$$
\varphi^{\circ}(x;\, v): = \limsup \limits_{\substack{ z\to x \\
						                   t \to 0^{+} }} \,\frac{\varphi(z + tv)-\varphi(z)}{t}.
$$
We are going to see that the set $\partial \varphi(x)$ is a pseudo-Jacobian of $\varphi$ at $x$. From the fundamental properties of the Clarke subdifferential (see e. g. \cite[Proposition 2.1.2]{Clarkebook}), we have that $\partial \varphi(x)$ is a nonempty, convex and $w^*$-compact subset of $X^*$, and for each $v\in X$,
$$
\varphi^{\circ}(x; v) = \max \{x^*(v) \, : \, x^*\in \partial \varphi(x)\}.
$$
From this, it is clear that
$$
\varphi'_{+}(x; v) \leq \max \{x^*(v) \, : \, x^*\in \partial \varphi(x)\}.
$$
Furthermore,
$$
\begin{aligned} \varphi'_{-}(x; v) &= \liminf_{t\to 0^+} \,\frac{\varphi(x+tv)-\varphi(x)}{t} \geq \liminf \limits_{\substack{ z\to x \\
			                                                           t \to 0^{+}}} \,\frac{\varphi(z)-\varphi(z-tv)}{t} \\
&= -\varphi^{\circ}(x; -v) \geq \min \{x^*(v) \, : \, x^*\in \partial \varphi(x)\}.
\end{aligned}$$
So, the conclusion follows from inequalities \eqref{pd-scalar}.}
\hfill $ \square$
\end{ex}

\begin{ex}\label{Thibault}
\rm{Let $f:U\subset X\to Y$ be a map, which is locally Lipschitz at a point $x\in U$. Note that, for every $y^*\in Y^*$, the composition $y^*\circ f$ is also locally Lipschitz at $x$. Following Thibault \cite{Thibault}, we say that a subset  $Jf(x) \subset {\mathcal L}(X, Y)$ is a \emph{Clarke-like generalized Jacobian} of $f$ at $x$ if the following conditions are satisfied:
\begin{enumerate}
\item[(i)] $Jf(x)$ is a nonempty, convex and {\tiny \emph{WOT}}-compact subset of  ${\mathcal L}(X, Y)$, and
\item[(ii)] For every $y^*\in Y^*$, we have that $y^* \circ Jf(x) = \partial (y^*\circ f)(x)$, where
$$
y^* \circ Jf(x):= \{y^* \circ T \, : \, T\in Jf(x) \}.
$$
\end{enumerate}
It is plain from Example \ref{Clarke} above that, in this case, $Jf(x)$ is a pseudo-Jacobian of $f$ at $x$.

An important example of Clarke-like generalized Jacobian, in the case where the target space is reflexive, is the generalized Jacobian $\partial f:X\to 2^{\mathcal{L}(X,Y)}$ considered by P\'ales and Zeidan in \cite{PZ-07} for every  map $f:U\subset X\to Y$ which is locally Lipschitz at a point $x\in U$.
Let us recall the definition of this mapping. Given a finite-dimensional linear subspace $L\subset X$, we say that $f$ is {\it $L$-G\^ateaux-differentiable} at a point $z\in U$ if there exists a continuous linear map $D_L(z):L\to Y$ such that
$$
\lim_{t \to 0} \frac{f(z+tv) - f(z)}{t}= D_L f(z)(v), \quad \text{for every} \quad v\in L.
$$
Denote by $\Omega_L(f)$ the set made up of all points $z\in U$ such that $f$ is $L$-G\^ateaux-differentiable at $z$, and let $\partial_L f(x)$ be the subset of $\mathcal{L}(L,\, Y)$ given by the the formula
$$
\partial_L f(x) := \bigcap_{\delta>0} \overline{co}^{\tiny WOT} \{ D_L f(z) \, : \, z\in B(x, \, \delta) \cap \Omega_L(f) \}.
$$
The \emph{P\'ales-Zeidan generalized Jacobian of $f$ at the point $x$} is the set
$$\partial f(x) = \left\{T\in \mathcal{L}(X,\, Y):\,\, T_{ \upharpoonright L }\in \partial_L f(x),\,\, \text{for each finite dimensional subspace}\, L\subset X\right\}.$$

Combining \cite[Theorem 3.8]{PZ-07} and \cite[Theorem 3.7]{PZ-07},  and taking into account that if $Y$ is a reflexive Banach space then $Y$ has the Radon-Nikod\'ym property  and the topology $\beta (X, V)=\beta(X, Y^*)$ considered in \cite{PZ-07} coincides with the {\tiny \emph{WOT}}-topology on ${\mathcal L}(X, Y)$, it follows that if $f:U\subset X\to Y$ is a locally Lipschitz map and $Y$ is reflexive, then the set $\partial f(x)$ satisfies conditions $(i)$ and $(ii)$ above for every $x\in U$. In particular, $\partial f$ is a pseudo-Jacobian mapping for $f$ on $U$. If $X$ and $Y$ are finite-dimensional, then $\partial f$ coincides with the usual Clarke generalized gradient (see \cite{Clarkebook}).
}
\hfill $ \square$
\end{ex}

\

Of special importance for us are upper semicontinuous pseudo-Jacobians.  A pseudo-Jacobian mapping $Jf:U\to 2^{\mathcal{L}(X,Y)}$ for a map $f:U\to Y$ is said to be {\it upper semicontinuous} at a point $x\in U$ if for every $\varepsilon>0$ there exists some $\delta>0$ such that $B(x,\delta)\subset U$ and $$Jf \bigl(B(x, \delta)\bigr) \subset Jf (x) + \varepsilon \cdot B_{\mathcal{L}(X,Y)},$$ where $Jf\bigl( B(x, \delta)\bigr) = \left\{T:\, T\in Jf(z),\, z\in B(x,\delta) \right\}.$ It is clear that if the map $f$ is $\mathcal{C}^1$-smooth, then the pseudo-Jacobian mapping $Jf = \{f'\}$ is upper semicontinuous at every $x\in U$. It is also well-known (see e.g. \cite{Clarkebook}) that if $X$ and $Y$ are finite-dimensional and the map $f$ is locally Lipschitz, then the Clarke generalized Jacobian $\partial f:X\to 2^{\mathcal{L}(X,Y)}$ is upper semicontinuous.  Next, we give another example of upper semicontinuous pseudo-Jacobian in the nonsmooth setting.

\begin{ex}\label{Loc-Lip}
\rm{
Let $f:U\subset X\to Y$ be a map, which is locally Lipschitz at a point $x\in U$, and let $Jf(x)$ be the closed unit ball centered at zero of radius ${\rm Lip} \, f (x)$ in the space of bounded linear operators ${\mathcal L}(X, Y)$, that is,
$$
Jf(x):= {\rm Lip}\, f (x) \cdot \overline{B}_ {{\mathcal L}(X, Y)}.
$$
Then, $Jf(x)$ is a  pseudo-Jacobian of $f$ at $x$. Moreover, if $Jf$ is locally Lipschitz on $U$, then the pseudo-Jacobian $Jf$ is is upper semicontinuous at every $x\in U$. Indeed, denote $K={\rm Lip}\, f (x)$, and fix nonzero vectors $v\in X$ and $y^*\in Y^*$. Given $\varepsilon>0$, choose $r>0$ such that $f$ is $(K + \varepsilon)$-Lipschitz on the ball $B(x, r)$.  If $0< t\Vert v \Vert <r$ then
$$
\vert y^* (f(x+tv)) - y^*(f(x))\vert \leq t (K + \varepsilon) \Vert y^* \Vert \cdot \Vert v \Vert.
$$
Thus,
$$
(y^* \circ f)'_{+}(x;v) \leq K\cdot \Vert y^* \Vert \cdot \Vert v \Vert.
$$
Next,  we will show that
$$
K\cdot \Vert y^* \Vert \cdot \Vert v \Vert = \sup\{\langle y^*, T(v) \rangle: \Vert T \Vert \leq K\}.
$$
Choose $v^*\in X^*$ with $\Vert v^* \Vert =1$ and $v^*(v) = \Vert v \Vert$. Given $\varepsilon>0$, pick $y\in Y$ such that $\Vert y\Vert \leq 1$ and $y^*(y) \geq \Vert y^* \Vert - \varepsilon$. Now, define the operator  $T:X \to Y$ by setting $$T(z)=K\cdot v^*(z)\cdot y, \quad z\in X.$$
Note that $\Vert T \Vert \leq K$ and, furthermore,
$$
\langle y^*, T(v)\rangle = \langle y^*, K \Vert v \Vert y\rangle \geq K \Vert v \Vert (\Vert y^*\Vert - \varepsilon).
$$
Hence, $Jf(x)$ is a pseudo-Jacobian of $f$ at $x$.

It remains to show that $Jf$ is upper semicontinuous. In order to see it,  fix $x\in U$ and $\varepsilon>0$, and choose $\delta>0$ such that
$$
\sup \left\{ \frac{\Vert f(u) - f(v) \Vert}{\Vert u -v \Vert}: \,  \,\,\, u, v\in B(x, \delta)\,\,\, \text{and} \,\,\, u\neq v \right\} < {\rm Lip}\, f (x) + \varepsilon.
$$
For each $z \in B(x, r)$, choosing $\delta'>0$ such that $B(z, \delta') \subset B(x, r)$ we have
$$
\sup \left\{ \frac{\Vert f(u) - f(v) \Vert}{\Vert u -v \Vert}: \,  \,\,\, u, v\in B(z, \delta')\,\,\, \text{and} \,\,\, u\neq v \right\} < {\rm Lip}\, f (x) + \varepsilon.
$$
Thus, ${\rm Lip}\, f (z) \leq {\rm Lip}\, f (x) + \varepsilon$ for all $z \in B(x, \delta)$. Hence, $Jf(z) \subset Jf(x) + \varepsilon \overline{B}_ {{\mathcal L}(X, Y)}$.
}
\hfill $ \square$
\end{ex}

The next result  provides a ``mean value property" for pseudo-Jacobians.

\begin{thm}[Mean value property]\label{MVT}
Let $X, Y$ be Banach spaces and $U$ an open convex subset of $X$. Suppose that  $f:U\to Y$ is  continuous and $Jf$ is a pseudo-Jacobian mapping for $f$ on $U$. Then, for each $u,v\in U$,
$$
f(v)-f(u) \in \overline{\co \,} \bigl(Jf([u,v])(v-u)\bigr).
$$
\end{thm}

\begin{proof}
Notice that the right-hand side above denotes the norm-closed convex hull in $Y$ of the set of all points of the form $T(v-u)$, where $T\in Jf(z)$ for some $z \in [u,v]$. For the proof suppose that, on the contrary, $f(v)-f(u) \notin \overline{\co \,} \bigl(Jf([u,v])(v-u)\bigr)$. Then, by the Hahn-Banach separation Theorem there exists a functional $y^*\in Y^*$ such that
\begin{equation}
\left\langle y^*, f(v)-f(u) \right\rangle > \sup \left\{ \left\langle y^*, y \right\rangle \, : \, y\in \overline{co}(Jf([u,v])(v-u))\right\}.\label{MTV-contradiction}
\end{equation}
Now consider the real-valued function $g:[0,1]\to \R$ defined as
$$
g(t) = \big\langle y^*, f(u + t(v-u)) - f(u) +t(f(u)-f(v)) \big\rangle.
$$
Note that $g$ is continuous on $[0,1]$ with $g(0) = g(1)$. Thus, the function $g$ attains its minimum or maximum at some point $t_0$ belonging to the open interval $(0,1)$. Suppose first that $t_0$ is a point of minimum. Then,
\begin{align*}
0 \, & \leq \, g'_{+}(t_0; 1) = \limsup_{t \to 0^+} \frac{1}{t} \big( g(t_0+t) - g(t_0)\big) \\
 & = \limsup_{t \to 0^+} \frac{1}{t} \big(\big\langle y^*, f(u+t_0(v-u)+t(v-u))-f(u+t_0(v-u)) + t(f(u)-f(v)) \big\rangle \big) \\
 & = (y^* \circ f)'_{+}(u+t_0(v-u); v-u) + \left\langle y^*, f(u)-f(v)\right\rangle.
\end{align*}
And therefore, using \eqref{pd-1},
\begin{align*}
\left\langle y^*, f(v)-f(u)\right\rangle & \leq (y^* \circ f)'_{+}(u+t_0(v-u); v-u) \\
& \leq \sup \left\{\left\langle y^*, T(v-u) \right\rangle \, : \, T\in Jf(u+t_0(v-u)) \right\}.
\end{align*}
Suppose now that $t_0$ is a point of maximum. Then,
\begin{align*}
0 \, & \geq \, g'_{-}(t_0; -1) = \liminf_{t \to 0^+} \frac{1}{t} \big( g(t_0 -t) - g(t_0)\big) \\
 & = \liminf_{t \to 0^+} \frac{1}{t} \big(\big\langle y^*, f(u+t_0(v-u)-t(v-u))-f(u+t_0(v-u)) - t(f(u)-f(v)) \big\rangle \big) \\
 & = (y^* \circ f)'_{-}(u+t_0(v-u); u-v) + \left\langle y^*, f(v)-f(u)\right\rangle.
\end{align*}
Hence, using now \eqref{pd-2},
\begin{align*}
\langle y^*, f(v)-f(u)\rangle & \leq - (y^* \circ f)'_{-}(u+t_0(v-u); u-v) \\
& \leq \sup \left\{- \langle y^*, T(u-v) \rangle \, : \, T\in Jf(u+t_0(v-u)) \right\} \\
& = \sup \left\{\langle y^*, T(v-u) \rangle \, : \, T\in Jf(u+t_0(v-u)) \right\}.
\end{align*}
Thus, in either case we obtain the inequality
\begin{align*}
\left\langle y^*, f(v)-f(u)\right\rangle & \leq \sup \left\{\left\langle y^*, T(v-u) \right\rangle \, : \, T\in Jf([u,v]) \right\} \\
& \leq \sup \left\{\langle y^*, y \rangle \, : \, y\in \overline{co}(Jf([u,v])(v-u)) \right\},
\end{align*}
which is a contradiction with \eqref{MTV-contradiction}.

\end{proof}

\

Let $f:U\subset X\to Y$ a map and $x\in U$. A pseudo-Jacobian mapping  $Jf: U \to 2^{{\mathcal L}(X, Y)}$ for $f$ on $U$ is said to be  \emph{locally bounded} (resp. \emph{locally  {\tiny WOT}-compact}, resp. \emph{locally {\tiny SOT}-compact}) at $x$, if there exists $r>0$ such that $B(x, r) \subset U$ and the set of operators
$Jf\bigl(B(x, r)\bigr)$ is bounded (resp. relatively {\tiny WOT}-compact, resp. relatively {\tiny SOT}-compact) in the space $\mathcal{L}(X,Y)$. We say that $Jf$ is \emph{locally bounded}   (resp. \emph{locally  {\tiny WOT}-compact}, resp. \emph{locally {\tiny SOT}-compact})  \emph{on} $U$ if this holds for every $x\in U$.

As a consequence of the mean value property, we obtain the following result.

\begin{cor}\label{local-lipschitness}
Let $f:U\to Y$ be a continuous map. Then, $f$ admits a locally bounded pseudo-Jacobian mapping on $U$ if, and only if, $f$ is locally Lipschitz on $U$.
\end{cor}

\begin{proof}
Suppose first that $Jf: U \to 2^{{\mathcal L}(X, Y)}$ is a locally bounded pseudo-Jacobian for $f$. Given $x\in U$ choose $L,\, r> 0$ such that such that the set $Jf(z)$ is contained in the closed ball $L\cdot \overline{B}_{{\mathcal L}(X, Y)}$, for each $z\in B(x, r)$. Fix $u, v \in B(x, r)$ and $\varepsilon>0$. From Theorem \ref{MVT} we can find $T\in \co (Jf([u, v])) \subset L\cdot \overline{B}_{{\mathcal L}(X, Y)}$ such that
$$
\Vert f(v) - f(u) - T(v-u) \Vert < \varepsilon,
$$
and hence, $\Vert f(v) - f(u) \Vert \leq L \, \Vert v-u \Vert + \varepsilon$. In this way we see that $f$ is $L$-Lipschitz on $B(x, r)$. The converse follows from Example \ref{Loc-Lip}.
\end{proof}

\

In our next example, we provide some examples of locally {\tiny \emph{WOT}}-compact pseudo-Jacobians, which will be useful later.

\begin{ex}\label{Loc-WOT}
Suppose that the map $f:U\subset X\to Y$ admits a factorization of the form
$$
f= T \circ g: U \to Z \to Y,
$$
where $Z$ is a reflexive Banach space, $g:U\to Z$ is locally Lipschitz on $U$, and $T\in {\mathcal L}(Z, Y)$.  For each $x\in U$, define
$$
J_1f(x):= {\rm Lip} \, g (x) \cdot \{ T\circ S \, : \, S \in \overline{B}_ {{\mathcal L}(X, Z)}\},
$$
and
$$
J_2f(x):= \{ T\circ S \, : \, S \in \partial \, g (x)\},
$$
where $\partial  g (x)$ is the P\'ales-Zeidan generalized Jacobian of $g$ at $x$ (see Example \ref{Thibault}). Then $J_1 f(x)$ and $J_2 f(x)$ are convex and  {\tiny \emph{WOT}}-compact pseudo-Jacobians for $f$ at $x$. Furthermore,  $J_1 f$ and $J_2 f$ are locally {\tiny \emph{WOT}}-compact pseudo-Jacobian mappings for $f$ on $U$.
\end{ex}
\begin{proof}
As explained in Examples \ref{Loc-Lip} and \ref{Thibault}, we have that ${\rm Lip}\, g (x) \cdot \overline{B}_ {{\mathcal L}(X, Z)}$ and $\partial g (x)$ are pseudo-Jacobians for $g$ at $x\in U$. It follows from the very definition that $J_1f(x)$ and $J_2f(x)$ are then pseudo-Jacobians for $f$ at $x$.

Moreover, since $Z$ is reflexive, we know that $\overline{B}_ {{\mathcal L}(X, Z)}$ is a {\tiny \emph{WOT}}-compact subset of ${\mathcal L}(X, Z)$. In addition, as we have remarked before, the topology $\beta (X, V) = \beta (X, Z^*)$ considered in \cite{PZ-07} coincides with the {\tiny \emph{WOT}}-topology in ${\mathcal L}(X, Z)$ by the reflexivity of $Z$. Thus, from \cite[Theorem 3.8]{PZ-07} we have that the set $\partial  g (x)$ is also convex and {\tiny \emph{WOT}}-compact in ${\mathcal L}(X, Z)$. Therefore, since the map $S \mapsto T\circ S$ is  {\tiny \emph{WOT}}-to-{\tiny \emph{WOT}} continuous from ${\mathcal L}(X, Z)$ to ${\mathcal L}(X, Y)$, we obtain that the sets  $J_1f(x)$ and  $J_2f(x)$ are convex and {\tiny \emph{WOT}}-compact in ${\mathcal L}(X, Y)$. Finally, since $g$ is locally Lipschitz, for each $x\in U$ there exist constants $r>0$ and $C>0$ such that for every $z\in B(x, r)$, $J_1f(z)$ and $J_2f(z)$ are contained in the set
$$
C\cdot \{ T\circ S \, : \, S \in \overline{B}_ {{\mathcal L}(X, Z)}\}.
$$
In this way we see that  $J_1 f$ and $J_2 f$ are locally {\tiny \emph{WOT}}-compact pseudo-Jacobian mappings for $f$ on $U$.
\end{proof}

\

We next give an optimality condition in terms of pseudo-Jacobians.

\begin{thm}[Optimality condition]\label{optimality}
Consider a real-valued function $\varphi:U\subset X \to \R$. Suppose that $\varphi$ attains a local extremum at a point $x_0\in U$, and  $J\varphi (x_0)$ is a pseudo-Jacobian for $\varphi$ at $x_0$. Then,
$$
0\in \overline{co}^{w^*} \bigl(J\varphi(x_0)\bigr).
$$
\end{thm}

\begin{proof}
It is enough to prove the statement when $\varphi$ attains a local minimum at $x_0$. Consider the set $H = \overline{co}^{w^*} \bigl(J\varphi(x_0)\bigr)$, that is, the $w^*$-closed convex hull of $J\varphi(x_0)$ in $X^*$. Let $\sigma_H : X \to (- \infty, +\infty]$ denote the corresponding support function of $H$, which is defined, for each $v\in X$, by
$$
\sigma_H (v) = \sup \left\{ \langle x^*, v \rangle \, : \, x^* \in H \right\}.
$$
Note that, given $x^* \in X^*$, we have that $x^* \in H$ if, and only if, $\langle x^*, v\rangle \leq \sigma_H (v)$ for every $v\in X$. Indeed, since $H$ is a convex and $w^*$-closed subset of $X^*$, by applying the Hahn-Banach separation theorem to the space $(X^*, w^*)$, and taking into account that the dual of this space coincides with $X$, we obtain that if $x^* \notin H$, then there exists some $v\in X$ such that $\langle x^*, v\rangle > \sigma_H (v)$.

Now, if $x_0$ is a local minimum of $\varphi$, we have that for each $v\in X$,
$$
0 \leq \varphi'_{+}(x_0; v)\leq \sup \left\{\langle x^*, v \rangle  \, : \, x^* \in J\varphi(x_0) \right\} \leq \sigma_H (v).
$$
Thus, $0\in H$.
\end{proof}

\begin{rem}
{\rm The set $\overline{co}^{w^*} \bigl(J\varphi(x_0)\bigr)$ in the previous statement is minimal, in the sense that we cannot replace the weak star closure by the closure in a stronger topology on $X^*$. Indeed, let $X$ be a dual Banach space, $X=Y^*$, take an element $\xi\in X^{*}=Y^{**}$, and consider the function $\varphi:X\to \mathbb R$ given by the formula
$$\varphi(x)= \|x\|-\left\langle \xi,\, x\right\rangle.$$
It is obvious that  $\min \limits_{x\in X} \varphi(x)=0=\varphi(0)$. Moreover, for every $v\in X$ we have
$$\frac{\varphi(tv)-\varphi(0)}{t}\to \|v\|-\left\langle \xi,\, v\right\rangle= \sup \left\{\left\langle y-\xi,\, v\right\rangle,\, y\in \overline{B}_Y\right\}, \quad \text{as} \quad t\to 0^+.$$
Thus, the set $J\varphi(0) := -\xi + \overline{B}_{Y}$ is a (norm-closed and convex) pseudo-Jacobian of $\varphi$ at $0$ (and hence, $0\in \overline{J\varphi(0)}^{w^*}$). However, if the space $X=Y^*$ is non-reflexive and $\xi\in Y^{**}\setminus Y$, then $0\not \in J\varphi(0)$.
\hfill $\square$
}
\end{rem}

\

Now, we shall establish some partial results concerning the chain rule for pseudo-Jacobians.
In the sequel, we consider Banach spaces $X$, $Y$ and $Z$, open subsets $U\subset X$ and $V\subset Y$, and continuous maps $f:U\to Y$ and $g:V\to Z$ with $f(U)\subset V$. If $Jf$ and $Jg$ are pseudo-Jacobian mappings for $f$ and $g$ on $U$ and $V$, respectively, and $x\in U$, then we denote by $Jg\left(f(x)\right)\circ Jf(x)$ the set made up of all compositions of the form $S\circ T$, where $S\in Jg\left(f(x)\right)$ and $T\in Jf(x).$ Our purpose is to find conditions to ensure that this set is a pseudo-Jacobian for $g\circ f$ at the point $x$. These conditions involve differentiability properties of $f$ or $g$ and boundedness or compactness properties of $Jg$ or $Jf$, respectively. The following is a first result in this direction.

\begin{prop}\label{chainrulebis} Let $f:U\to V$ and $g:V\to Z$ be continuous maps, let $Jg$ be a pseudo-Jacobian mapping of $g$ and $x\in U$. If
$f$ is G\^ateaux differentiable at $x$ and $Jg$ is locally bounded at $f(x)$, then the set
$Jg\left(f(x)\right)\circ f'(x)$
is a pseudo-jacobian of $g\circ f$ at $x$.
\end{prop}

\begin{proof}
We may assume that $U=X$ and $V=Y$. Pick $z^*\in Z^*$ and $v\in X$, and set, for each $t> 0$, $\Delta(t) = (f(x+tv)-f(x))/t$. Since $Jg$ is locally bounded at $f(x)$, Corollary
\ref{local-lipschitness} ensures the existence of constants $L,r> 0$ such that
$$\left\langle z^*,\, g\left(f(x)+y_1\right)-g\left(f(x)+y_2\right)\right\rangle\leq L\|y_1-y_2\|, \quad \text{whenever} \quad y_1,y_2\in r\overline{B}_Y.$$
Thus, if $t> 0$ is small enough, we have
$$\left\langle z^*,\, g\left(f(x)+t\Delta(t)\right)-g\left(f(x)+tf'(x)(v)\right)\right\rangle\leq Lt\|\Delta(t)-f'(x)(v)\|,$$ and consequently,
$$\begin{aligned} \left\langle z^*,\, \frac{g\left(f(x+tv)\right)-g\left(f(x)\right)}{t}\right\rangle &= \frac{1}{t}\left\langle z^*,\, g\left(f(x)+t \Delta(t)\right)-g\left(f(x)+tf'(x)(v)\right)\right\rangle\\ &+\left\langle z^*,\,\frac{g\left(f(x)+tf'(x)(v)\right)-g\left(f(x)\right)}{t}\right\rangle\\ &\leq L\|v\|\left\|\Delta(t)-f'(x)(v)\right\|\\ &+ \left\langle z^*, \frac{g\left(f(x)+tf'(x)(v)\right)-g\left(f(x)\right)}{t}\right\rangle.\end{aligned}$$
Since $f$ is G\^ateaux differentiable at $x$ we have $\|\Delta(t)-f'(x)(v)\|\to 0$ as $t\to 0$, and hence,
$$\begin{aligned} (z^*,\, g\circ f)'_{+}(x;\, v) &\leq \limsup_{t\to 0^+} \left\langle z^*, \frac{g\left(f(x)+tf'(x)(v)\right)-g\left(f(x)\right)}{t}\right\rangle\\ &\leq  \sup \left\{\left\langle z^*,\, \left(T\circ f'(x)\right)(v)\right\rangle:\, T\in Jg\left(f(x)\right)\right\}.\end{aligned}$$
\end{proof}

As a consequence of the former proposition, we obtain the following result.

\begin{cor}
Let $f:U\subset X\to Y$ and $g:Y\to Z$ be continuous maps and $x\in U$. If $f$ is G\^ateaux differentiable at $x$ and $g$ is weakly
G\^ateaux differentiable and locally Lipschitz at $f(x)$, then the map $g\circ f$ is weakly G\^ateaux differentiable at $x$, and $(g\circ f)'(x)=g'\left(f(x)\right)\circ f'(x)$.
\end{cor}

\begin{proof}
The previous proposition guarantees that the map $g\circ f$ admits a psuedo-Jacobian at $x$ made up of a unique operator, namely,
$g'\left(f(x)\right)\circ f'(x)$. Thus, according to Example \ref{wG} we have that $g\circ f$ is weakly G\^ateaux differentiable at $x$, and $(g\circ f)'(x)= g'\left(f(x)\right)\circ f'(x)$.
\end{proof}

Next, we establish a chain rule for the composition $g\circ f$ assuming differentiability properties of  $g$. If $g:V\to Z$  is weakly G\^ateaux differentiable on all of $V$, we say that the derivative $g'$ is \emph{Hadamard continuous} (resp.  \emph{weak Hadamard continuous}) at a point $y_0\in V$, if
for every compact (resp.  weakly compact) subset $K\subset Y$, we have $\sup_{h\in K} \left\|\left\langle g'(y)-g'(y_0),\, h\right\rangle\right\| \to 0$ as $y\to y_0.$ To simplify the statement of our next result, it is also convenient to introduce the following definition.

\begin{defn}
The pseudo-Jacobian mapping $Jf:U\to 2^{\mathcal{L}(X,Y)}$ is called \emph{directionally bounded} (resp. \emph{directionally  weakly compact}, resp. \emph{directionally compact}) at a point $x\in U$, if for every $v\in X$ there exists $r> 0$ such that
the set
$$
Jf\left([x,\, x+rv]\right) (v) := \left\{T(v) \, : \, T\in Jf(z),\, z\in [x,\, x+rv]\right\}
$$
is bounded (resp. relatively weakly compact, resp. relatively norm-compact) in the space $Y$. If this property holds for every $x\in U$, then we say that the pseudo-Jacobian mapping $Jf$ is directionally bounded (resp. directionally weakly compact, resp. directionally compact) on $U$.
\end{defn}

It is clear that if the pseudo-Jacobian mapping $Jf:U\to 2^{\mathcal{L}(X,Y)}$ is  locally bounded (resp. locally  {\tiny \emph{WOT}}-compact, resp. locally {\tiny \emph{SOT}}-compact) on $U$,  then $Jf$ is directionally bounded (resp. directionally weakly compact, resp. directionally compact) on $U$.

\begin{thm}\label{chainrule}
Let $f:U\to V$ be a continuous map and $Jf$  a pseudo-Jacobian mapping of $f$ on $U$. Suppose that  $g:V\to Z$ is a continuous and weakly G\^ateaux differentiable map, and consider $x\in U$. Then, the set  $g'\left(f(x)\right)\circ Jf(x)$ is a pseudo-Jacobian of $g\circ f$ at $x$ if one of the following conditions hold:
\begin{enumerate}
\item[(i)] $Jf$ is directionally bounded at $x$ and $g'$ is norm-continuous at $f(x)$.
\item[(ii)] $Jf$ is directionally compact at $x$ and $g'$ is Hadamard continuous at $f(x)$.
\item[(iii)] $Jf$ is directionally weakly compact at $x$ and $g'$ is weak Hadamard continuous at $f(x)$.
\end{enumerate}
\end{thm}

\begin{proof}
We can assume that $U=X$ and $V=Y$. Pick a functional $z^*\in Z^*$ and a vector $v\in X$, and let
$(t_n)_n\subset (0,\infty)$ such that $t_n\to 0$, and
$$
\lim_n \left\langle z^*,\, \frac{g\left(f(x+t_nv)\right)-g\left(f(x)\right)}{t_n} \right\rangle= (z^*\circ g\circ f)'_{+}(x;\, v).
$$
According to Theorem \ref{MVT}, there exists a sequence $(y_n)_n\subset Y$ such that $y_n\in \left[f(x),\, f(x+t_n v)\right]$ and
$$
g(f(x+t_n v))-g(f(x)) = g'(y_n)\left(f(x+t_n v)-f(x)\right), \quad \text{for all} \quad n\in \mathbb N.
$$ Let us write, for each $n\in \mathbb N$, $\Delta_n = \frac{1}{t_n}\left(f(x+t_n v)-f(x)\right)$. Then,
\begin{equation}
\left\langle z^*,\, \frac{g\left(f(x+t_n v)\right)-g\left(f(x)\right)}{t_n} \right\rangle = \left\langle z^*,\, \left(g'(y_n)-g'(f(x))\right)\left(\Delta_n\right)\right\rangle+\left\langle z^*\circ g'(f(x)),\, \Delta_n\right\rangle,\nonumber
\end{equation}
and bearing in mind that $Jf(x)$ is a pseudo-Jacobian of $f$ at $x$, it follows that
$$
\begin{aligned} \limsup_n \left\langle z^*,\, \frac{g\left(f(x+t_n v)\right)-g\left(f(x)\right)}{t_n} \right\rangle &\leq \limsup_n \left\langle z^*,\, \left(g'(y_n)-g'(f(x))\right)\left(\Delta_n\right)\right\rangle\\ &+\limsup_n \left\langle z^*\circ g'(f(x)),\, \Delta_n\right\rangle\\ &\leq \limsup_n \left\langle z^*,\, \left(g'(y_n)-g'(f(x))\right)\left(\Delta_n\right)\right\rangle\\ &+\sup_{T\in Jf(x)} \left\langle z^*\circ g'(f(x))\circ T, v\right\rangle.
\end{aligned}
$$
Thus, it suffices to show that
\begin{equation}
\limsup_n \left|\left\langle z^*,\, \left(g'(y_n)-g'(f(x))\right)\left(\Delta_n\right)\right\rangle\right|=0.\label{chainrule-eq1}
\end{equation}
Notice that, because of Theorem \ref{MVT}, we have
\begin{equation}
\Delta_n\in \overline{co}\, \bigl(Jf([x,\, x+t_n v])(v)\bigr).\label{chainrule-eq2}
\end{equation}
Suppose first that condition $(i)$ holds. Then, there exists a number $r> 0$ such that the set $Jf([x,\, x+rv])(v)$ is bounded in $Y$, and therefore $K=\overline{co}\, \bigl(Jf([x,\, x+rv])(v)\bigr)$ is bounded as well. Since $t_n \to 0$, by \eqref{chainrule-eq2} we get $\Delta_n \in K$ for $n$ large enough. Moreover, as $g$ is continuously Fr\'echet differentiable at $f(x)$ it follows that $\left\|g'(y_n)-g'(f(x))\right\|\to 0$. Thus,
$$
\left\|\left\langle g'(y_n)-g'(f(x)),\, \Delta_n\right\rangle \right\|\to 0.
$$
This proves \eqref{chainrule-eq1} under condition $(i)$.

Now, suppose that assertion $(ii)$ is satisfied. There exists $r> 0$ such that the set $Jf([x,\, x+rv])(v)$  is relatively norm-compact in $Y$. Hence, the set $K=\overline{co}\, \bigl(Jf([x,\, x+rv])(v)\bigr)$ is norm-compact in $Y$. Using again that  $t_n \to 0$ and \eqref{chainrule-eq2}, we have  that for $n$ large enough,  $\Delta_n \in K$,  and therefore
$$
\begin{aligned} \left|\left\langle z^*,\, \left(g'(y_n)-g'(f(x))\right)\left(\Delta_n\right)\right\rangle\right|&\leq \sup_{h\in K} \left|\left\langle z^*,\, \left(g'(y_n)-g'(f(x))\right)(h)\right\rangle\right|\\ &\leq \|z^*\|\sup_{h\in K} \left\|\left\langle g'(y_n)-g'(f(x)),\, h\right\rangle\right\|.\end{aligned}
$$
Since  $g'$ is Hadamard continuous at $f(x)$, it follows that
$$
\sup_{h\in K} \left\|\left\langle g'(y_n)-g'(f(x)),\, h\right\rangle\right\|\to 0.
$$ This and the former inequality imply \eqref{chainrule-eq1}.

Finally, the proof of our claim under hypothesis $(iii)$ is completely analogous, taking into account that, by Krein-Smulian theorem, the closed convex hull of a weakly compact set in a Banach space is weakly compact as well.

\end{proof}

\

We end this section by introducing a property of pseudo-Jacobian mappings which will be useful to achieve our results concerning invertibility of nonsmooth maps in the next section. This property is based on the validity of the chain rule for the composition with distance functions,  together with a (rather technical) stability condition on the pseudo-Jacobian of such compositions. In what follows, given a pseudo-Jacobian mapping $Jf$ for a map $f:U\subset X\to Y$, and two points $x\in U$ and $y\in Y$ with $y\neq f(x)$, we will denote by $A_{x,y}(f)$ the subset of $X^*$ defined by the formula
$$A_{x,y}(f)= \partial d_y\left(f(x)\right)\circ \overline{co} \left(Jf(x)\right),
$$
where the symbol $d_y$ stands for the distance function to $y$, that is $d_y(v)=\|v-y\|, \, v\in Y$.

\

\begin{defn}\label{crc}
Let $f:U\subset X\to Y$ be a continuous map and $Jf$ a pseudo-Jacobian mapping for $f$ on $U$. We say that $Jf$ satisfies the chain rule condition on $U$ if, for each $x\in U$ and each $y\in Y$ with $y\neq f(x)$, the set $A_{x,y}(f)$ is a $w^*$-closed and convex pseudo-Jacobian of the function $d_y\circ f$ at the point $x$.
\end{defn}

The following result provides a first example of pseudo-Jacobian satisfying the former condition.

\begin{prop}\label{crc-ex1}
If $f:U\subset X\to Y$  is continuous and G\^ateaux differentiable on all of $U$, then the pseudo-Jacobian $Jf(x) = \{f'(x)\},\, x\in U$, satisfies the chain rule condition.
\end{prop}
\begin{proof}
Let us fix $x\in U$ and let $y\in Y$ with $y\neq f(x)$. Since $f$ is G\^ateaux differentiable we have
$$
A_{x,y}(f) =  \partial \|\cdot\|(f(x)-y)\circ f'(x),
$$
and bearing in mind that $\partial \|\cdot\|(f(x)-y)$ is a $w^*$-compact convex subset of $Y^*$ (c.f. \cite[p. 7]{Ph}), it follows that $A_{x,y}(f)$ is $w^*$-compact and convex as well. In particular, the set $A_{x,y}(f)$ is $w^*$-closed and convex. Further, since the norm $\|\cdot\|$ is Lipschitz, we have that $\partial \|\cdot\|$ is locally bounded at the point $f(x)-y$. Applying Proposition \ref{chainrulebis} we deduce that $A_{x,y}(f)$ is a pseudo-Jacobian of the function $d_y\circ f$ at $x$.
\end{proof}

As a consequence of Theorem \ref{chainrule} we obtain our next result, which provides examples of pseudo-Jacobian mappings satisfying the chain rule condition under some compactness assumptions on such pseudo-Jacobians and smoothness properties of the norm of the target space.

\begin{cor}\label{crc-ex}
Let $f:U\subset X\to Y$ be a continuous map, and let $Jf$ be a pseudo-Jacobian mapping for $f$ on $U$ such that, for every $x\in U$, the set $Jf(x)$ is convex and {\tiny \emph{WOT}}-compact. Then, $Jf$ satisfies the chain rule condition in any of the following cases:
\begin{enumerate}
\item[(i)] The norm of $Y$ is Fr\'echet smooth and for every $x\in U$, $Jf$ is directionally bounded at $x$.
\item[(ii)] The norm of $Y$ is G\^ateaux smooth and for every $x\in U$, $Jf$ is directionally compact at $x$.
\item[(iii)] The norm of $Y$ is weak Hadamard smooth and for every $x\in U$, $Jf$ is directionally weakly compact at $x$.
\end{enumerate}
\end{cor}
\begin{proof}
Pick $x\in U$, and let $y\in Y$ with $f(x)\neq y$. By the smoothness of the norm, there is a unique norm-one functional
$y^{\ast}\in Y^{\ast}$ such that $\left\langle y^{\ast}, f(x) - y\right\rangle = d_y\left(f(x)\right)$. Thus,
$$
A_{x,y}(f)= y^*\circ Jf(x),
$$
and taking into account that $Jf(x)$ is convex and {\tiny \emph{WOT}}-compact it follows that $A_{x,y}(f)$ is a convex and $w^*$-compact subset of $X^*$.

It remains to show that $A_{x,y}(f)$ is a pseudo-Jacobian for the function $d_y\circ f$ at the point $x$. Assume first that the norm $\|\cdot\|$ on $Y$ is Fr\'echet smooth. According to a classical result by Smulyan (see e.g. \cite[Corollary 7.22]{Fabian}) it follows that the derivative of $\|\cdot\|$ is norm-to-norm continuous on $Y\setminus \{0\}$. Therefore, thanks to  Theorem \ref{chainrule} (i), we have that the set $y^*\circ Jf(x)$ is a pseudo-Jacobian of $d_y\circ f$ at $x$. The proof under the other conditions is similar. If the norm of $Y$ is G\^ateaux smooth, then its derivative is norm-to-weak$^{*}$ continuous on $Y\setminus \{0\}$. An easy argument yields then that $\|\cdot\|'$ is Hadamard continuous on $Y\setminus \{0\}$, and Theorem \ref{chainrule} (ii) applies. Finally, if the norm on $Y$ is weak Hadamard smooth, then its derivative is weak Hadamard continuous on $Y\setminus \{0\}$, and we can use Theorem \ref{chainrule} (iii).

\end{proof}

\section{Inversion of nonsmooth maps}

Our main aim in this section is to establish several criteria to guarantee that a continuous map $f$ from an open subset set of a Banach space into a Banach space is (locally or globally) invertible when a pseudo-Jacobian mapping of $f$ satisfying the chain rule condition is at hand. We first obtain an open mapping theorem, using the {\it surjectivity index} introduced by Rabier in \cite{Rabier}. Recall that if $T: X \to Y$ is a bounded linear operator, then the \emph{co-norm} of $T$ is the number
$$/\hspace{-0.9mm}/T/\hspace{-0.9mm}/= \inf_{\|x\|=1} \|Tx\|.$$
The \emph{surjectivity index} of $T$ is defined as $$\nu(T) = /\hspace{-0.9mm}/\, T^*/\hspace{-0.9mm}/,$$ where $T^*:Y^*\to X^*$ denotes the adjoint operator of $T$. It is clear that $T$ is surjective if, and only if, $\nu(T)> 0$. Observe also that if $T$ has a right inverse $R$
(that is, if there exists an operator $R\in \mathcal{L}(Y,X)$ such that $T\circ R = {\rm Id}_Y$), then
$$1 =\nu(R^*\circ T^*) = \inf_{y^*\in S_{Y^*}} \sup_{y\in S_Y} \left\|\left\langle T^* y^*, R y\right\rangle\right\|\leq \|R\|\inf_{y^*\in S_{Y^*}} \|T^* y^*\| = \|R\| \nu(T),$$ and therefore, $\nu(T) \geq \|R\|^{-1}$.
Analogously, if $T$ admits a left inverse $L: Y \to X$, then
$$1 = \sup_{x^*\in S_{X^*}} \left\|T^*(L^* x^*)\right\|\geq /\hspace{-0.9mm}/\, T^*/\hspace{-0.9mm}/ \sup_{x^*\in S_{X^*}} \|L^* x^*\| = \nu(T) \|L^*\|= \nu(T) \|L\|,$$
that is, $\nu(T) \leq \|L\|^{-1}$. In particular, for any isomorphism $T:X\to Y$ we have $$\nu(T) = \left\|T^{-1}\right\|^{-1} = /\hspace{-0.9mm}/T/\hspace{-0.9mm}/.$$
Our open mapping theorem reads as follows.
\begin{thm}\label{main-open}
Let $U$ be an open subset of $X$, let $f:U\to Y$ be a continuous map and $Jf$ a pseudo-Jacobian mapping for $f$ on $U$ satisfying the chain rule condition. Suppose that
$$
\alpha:=\inf \{ \nu(T) \,  : \, T\in  \co \, Jf(x); \, x \in U \}> 0.
$$
Then, for each open ball $B(x_0,\, \delta)\subset U$ we have
\begin{equation}
B\bigl(f(x_0),\, \delta \alpha\bigr)\subset f\bigl(\B(x_0,\, \delta)\bigr).\label{open}
\end{equation}
Therefore,  the set $V:=f(U)$ is open in $Y$, and $f$ is an open map from $U$ to $V$.
\end{thm}
\begin{proof}
Consider an open ball $B(x_0,\, \delta)\subset U$ and let $0<\eta <\delta$. Fix a vector $y\in \B\bigl(f(x_0),\, \eta \alpha\bigr)$, and
choose numbers $\varepsilon,\, \lambda> 0$ such that
$$\frac{\varepsilon}{\eta}< \lambda < \alpha \quad \text{and} \quad \|f(x_0)-y\|< \varepsilon.$$
Let $\varphi:X\to \mathbb R \cup \{+ \infty \}$ be the function defined as
$\varphi(x)= \|f(x)-y\|$ if $x\in \overline{B}(x_0, \, \eta)$, and $\varphi(x)=+\infty$ otherwise. It is easy to see that $\varphi$ is continuous on $B(x_0, \, \eta)$ and lower semicontinuous  on all of $X$. Moreover, $$\varphi(x_0)< \inf
\limits_{x\in X} \varphi(x)+\varepsilon.$$ An appeal to Ekeland variational principle (c.f. \cite[p. 45]{Ph}) yields the existence of a vector $z\in X$ such that
\begin{enumerate}
\item[(a)] $\|z-x_0\|\leq \frac{\varepsilon}{\lambda}< \eta$, and
\item[(b)] $\varphi(z)< \varphi(x)+\lambda \|z-x\|$, for all $x\neq z.$
\end{enumerate}
We shall show that $f(z)=y$. Assume the contrary. Because of $(b)$, the function $$\psi(x)= \varphi(x)+\lambda \|z-x\|, \quad x\in X,$$ attains a (global) minimum at the point $z$. On the other hand, thanks to inequality $(a)$ we have $z\in B(x_0, \, \eta)\subset U$, and bearing in mind that $Jf$ satisfies the chain rule condition on $U$, it follows that the set $A_{z,y}(f)=\partial \|\cdot\|\left(f(z)-y\right)\circ \overline{{\co }}\, \left(Jf(z)\right)$  is a convex $w^*$-closed pseudo-Jacobian of $\varphi$ at $z$. Moreover, since the norm on $X$ is $1$-Lipschitz, it follows that the closed unit ball $\overline{B}_{X^*}$ is a convex $w^*$-compact pseudo-Jacobian of this norm at every point. Therefore, the set
$$
J\psi(z) := A_{z,y}(f)+ \lambda \overline{B}_{X^*}
$$
is a convex $w^*$-closed pseudo-Jacobian of the function $\psi$ at $z$. And taking into account that $\psi$ attains a minimum at this point, thanks to Theorem \ref{optimality} we obtain $0\in J\psi (z)$.

Thus, given any number $0< r < 1$, we can find an operator $T\in \co\, Jf(z)$ and functionals
$x^*\in \overline{B}_{X^*}$ and $y^*\in \partial \|\cdot\|\left(f(z)-y\right)$ such that
\begin{equation}
\left\|y^*\circ T+ \lambda x^*\right\|< r.
\end{equation}
Hence,
$$
\alpha \leq \nu(T) \leq \Vert T^* (y^*) \Vert \leq \Vert y^* \circ T + \lambda x^* \Vert + \lambda \Vert x^* \Vert < r + \lambda.
$$
Letting $r$ go to zero we get $\alpha \leq \lambda$, which is a contradiction with our choice of $\lambda$. Consequently, $f(z)=y$. In this way we see that
$B\bigl(f(x_0),\, \eta \alpha\bigr)\subset f\bigl(\B(x_0,\, \eta)\bigr)$ for every $0< \eta < \delta$. As a consequence, the inclusion \eqref{open} is proved. It follows that $V=f(U)$ is open in $Y$, and $f$ is an open map.
\end{proof}

As a first consequence of the former theorem, we obtain the aforementioned result by Ekeland.
\begin{cor}$($\cite[Theorem 2]{Ekeland}$)$\label{Ekeland}
Let $U$ be an open subset of $X$ and $f:U\to Y$ be a continuous map. Assume that $f$ is G\^ateaux differentiable on all of $U$ and, for every $x\in U$, the derivative $f'(x)$ admits a right inverse $R(x)\in \mathcal{L}(Y, X)$. If
$$\kappa := \sup_{x\in U} \|R(x)\|< \infty$$
then for every open ball $B(x_0, \delta)\subset U$ we have
$$B\bigl(f(x_0), \, \delta \kappa^{-1}\bigr)\subset f\bigl(B(x_0, \delta)\bigr).$$
In particular, $f(U)$ is an open subset of $V$, and $f$ is an open map.
\end{cor}
\begin{proof} Consider the pseudo-Jacobian mapping $Jf:U\to 2^{\mathcal{L}(X,Y)}$ defined by the formula $Jf(x) = \{f'(x)\}, \, x\in U$. According to Proposition \ref{crc-ex1} we have that $Jf$ satisfies the chain rule condition on $U$. On ther hand, because of the hypothesis we get
$$\alpha:=\inf_{x\in U}\nu\left(f'(x)\right)\geq \inf_{x\in U} \|R(x)\|^{-1}\geq \kappa^{-1}> 0,$$ and Theorem \ref{main-open} applies.
\end{proof}

Now, we present some inversion theorems. In order to simplify their statements, we introduce the following concept.

\begin{defn}
Let $f: U\subset X \to Y$ be a map, and let $Jf$ be a pseudo-Jacobian mapping for  $f$.  The regularity index of $Jf$ at a point $x\in U$ is defined as the number
$$
\alpha_{Jf}(x):=
\sup_{r>0} \inf\{/\hspace{-0.9mm}/T/\hspace{-0.9mm}/: T \in co\, Jf(B(x, r))\}.
$$
We say that the pseudo-Jacobian mapping $Jf$ is regular at a point $x\in U$ if, for some $r>0$,  every operator $T\in \co\, Jf(B(x, r))$ is an isomorphism,  and furthermore $\alpha_{Jf}(x)>0$.
\end{defn}

The character of regularity of the pseudo-Jacobian mapping $Jf$ at a point $x$, as well of the regularity index $\alpha_{Jf}(x)$, are of uniform nature: they rely on the behaviour of $Jf$ in a neighborhood of $x$. The next result shows that in the case that $Jf$ is upper semicontinuous at $x$, then both characteristics depend only on the behaviour of $Jf$ at that point.

\begin{prop}\label{usc-2}
Let $U$ be an open subset of $X$, let $f:U\to Y$ be a continuous map and $Jf$ a pseudo-Jacobian mapping for $f$ on $U$, which is upper semicontinuous at a point $x\in U$. Assume that:
\begin{enumerate}
\item[(i)] Every $T \in {\rm co}\, Jf(x)$ is an isomorphism, and
\item[(ii)] $\inf \left\{/\hspace{-0.9mm}/T/\hspace{-0.9mm}/: T \in co\, Jf(x))\right\}> 0.$
\end{enumerate}
Then, $Jf$ is regular at $x$ and $\alpha_{Jf}(x) = \inf \left\{/\hspace{-0.9mm}/T/\hspace{-0.9mm}/: T \in co\, Jf(x))\right\}$.
\end{prop}
\begin{proof}
Let us write $\beta := \inf \left\{/\hspace{-0.9mm}/T/\hspace{-0.9mm}/: T \in co\, Jf(x))\right\}$ and take $0 < \varepsilon < \beta$. The upper semicontinuity of $Jf$ ensures the existence of a number $r> 0$ such that
$$Jf \bigl(B(x, r)\bigr) \subseteq Jf (x) + \varepsilon B_{\mathcal{L}(X,Y)} \subseteq {\rm co}\, Jf(x) + \varepsilon B_{\mathcal{L}(X,Y)},$$
and taking into account that the latter set is convex we get
\begin{equation}
{\rm co}\, Jf\bigl(B(x, r)\bigr)\subseteq {\rm co}\, Jf(x) + \varepsilon B_{\mathcal{L}(X,Y)}.\nonumber
\end{equation}
Now, let $T\in {\rm co}\, Jf\bigl(B(x, r)\bigr)$. Because of the previous inclusion, we can find an operator $T_0\in Jf(x)$ such that $\|T-T_0\|< \varepsilon$. By hypothesis, $T_0$ is an isomorphism, and $$\|T_0^{-1}\| = /\hspace{-0.9mm}/T_0/\hspace{-0.9mm}/^{-1} \leq \beta^{-1}.$$
Therefore,
$$\left\|T_0^{-1}\circ T- {\rm Id}_X \right\| = \left\|T_0^{-1}\circ (T-T_0)\right\|\leq \varepsilon \beta^{-1}< 1.$$
In particular, $T_0^{-1}\circ T$ is an isomorphism, and thus, $T$ is an isomorphism as well.

It remains to show that $\alpha_{Jf}(x) = \beta$. It is clear that $\alpha_{Jf}(x) \leq \beta$. As before, we fix $0< \varepsilon < \beta$, and find a number $r> 0$ with the property that for any $T\in Jf\bigl(B(x, r)\bigr)$ there is $T_0\in {\rm co}\, \bigl(Jf(x)\bigr)$
satisfying $\|T-T_0\|< \varepsilon$. Thus,
$$
/\hspace{-0.1cm}/T/\hspace{-0.1cm}/ = \inf \limits_{h\in S_X} \|Th\|\geq \inf \limits_{h\in S_X} \|T_0 h\|- \varepsilon = /\hspace{-0.1cm}/T_0/\hspace{-0.1cm}/-\varepsilon \geq \alpha- \varepsilon.
$$
Consequently, $\alpha_{Jf}(x)\geq \alpha-\varepsilon$, for all $\varepsilon> 0$, and therefore, $\alpha_{Jf(x)}\geq \alpha$.
\end{proof}

\

As a consequence of the former proposition it follows that if the map $f:U\subset X\to Y$ is $\mathcal{C}^{1}$-smooth on $U$, then the pseudo-Jacobian mapping $Jf = \{f'\}$ is regular at a point $x\in U$ if, and only if, $f'(x)$ is an isomorphism. Moreover, in this case $\alpha_{Jf}(x)= \|f'(x)^{-1}\|^{-1}$.

\

Next, we present a local inversion result in terms of pseudo-Jacobians.
\begin{thm}\label{local-inversion}
Let $U$ be an open subset of $X$, let $f:U\to Y$ be a continuous map and $Jf$ a pseudo-Jacobian mapping for $f$ on $U$ satisfying the chain rule condition. If $Jf$ is regular at a point $x\in U$, then for every $0< \alpha < \alpha_{Jf}(x)$ there exists an open ball $B(x, r)\subseteq U$ with the following properties:
\begin{enumerate}
\item The set $V:= f\bigl(B(x, r)\bigr)$ is open in $Y$,
\item The map $f_{\upharpoonright B(x, r)}: B(x, r) \to V$ is an homeomorphism, and
\item The map $(f_{\upharpoonright B(x, r)})^{-1}$ is $\alpha^{-1}$-Lipschitz on $V$.
\end{enumerate}
\end{thm}

The main ingredient in the proof of this theorem is the following lemma.
\begin{lema}\label{main-inverse}
Let $U$ be an open convex subset of $X$, let $f:U\to Y$ be a continuous map and $Jf$ a pseudo-Jacobian mapping for $f$ on $U$ satisfying the chain rule condition. Suppose that every operator in $co\, Jf(U)$ is an isomorphism and, furthermore,
$$
\alpha:= \inf\{/\hspace{-0.9mm}/T/\hspace{-0.9mm}/: T \in co\, Jf(U)\} > 0.
$$
Then, the set $V:=f(U)$ is open in $Y$, and $f$ is an homeomorphism from $U$ to $V$, whose inverse is $\alpha^{-1}$-Lipschitz on $V$.
\end{lema}

\begin{proof}
Note that, since  every $T\in co\, Jf(U)$ is a linear isomorphism, we have $\nu(T) = /\hspace{-0.1cm}/T/\hspace{-0.1cm}/$. Then, from the Theorem \ref{main-open} we obtain that $V$ is an open set and $f$ is an open map. Now, we shall show that
\begin{equation}
\|f(x_1)-f(x_2)\|\geq \alpha \|x_1-x_2\|, \quad \text{whenever} \quad x_1,x_2\in U.\label{lipschitz}
\end{equation}
Take $x_1,x_2\in U$, and pick a number $0< r < 1$. According to Theorem \ref{MVT}, we can find an operator $T\in \co\, \left(Jf([x_1,x_2])\right)$ such that
$$\|f(x_1)-f(x_2)-T(x_1-x_2)\|\leq r\alpha \|x_1-x_2\|.$$
Since the set $U$ is convex we have $[x_1,x_2]\subset U$, and thus $/\hspace{-0.1cm}/T/\hspace{-0.1cm}/\geq \alpha$ . So, from the previous inequality it follows that
$$\begin{aligned} \|f(x_1)-f(x_2)\|&\geq \|T(x_1-x_2)\|-r\alpha\|x_1-x_2\|\\ &\geq /\hspace{-0.1cm}/T/\hspace{-0.1cm}/\|x_1-x_2\|-r\alpha\|x_1-x_2\|\\ &\geq \alpha(1-r)\|x_1-x_2\|.\end{aligned}$$
As the number $r$ was chosen arbitrarily, from this inequality we obtain \eqref{lipschitz}.
Now, it follows at once from \eqref{lipschitz} that $f$ is one-to-one on $U$, and that its inverse is $\alpha^{-1}$-Lipschitz on $V$.
\end{proof}

\emph{Proof of Theorem \ref{local-inversion}.} The regularity of $Jf$ ensures the existence of a number $r> 0$ with $B(x, r)\subset U$, and such that every operator $T\in {\rm co}\, Jf(U)$ is an isomorphism and
$\inf \left\{\,\hspace{-0.9mm}/\hspace{-0.9mm}/T/\hspace{-0.9mm}/:\, T\in {\rm co}\, Jf(U)\right\}\geq \alpha.$
Thus, the conclusion follows applying Lemma \ref{main-inverse} to the map $f_{\upharpoonright B(x, r)}$. \hfill $\square$

\

Now, we establish a criterion on global invertibility involving an Hadamard-like integral condition in terms of the regularity index of a pseudo-Jacobian. First, we recall that if $X, Y$ are Banach spaces,  $U\subset X$ is an  open set and $f:U \to Y$  a continuous map, the {\it lower scalar Dini derivative} of $f$ at a point $x\in U$  is defined  as
$$
D^{-}_x f =\liminf_{z\to x} \frac{\|f(z) - f(x)\|}{\|z - x\|}.
$$

It is a classical result due to F. John (see Corollary in page 91 of \cite{John}) that, if $f:X \to Y$ is a local homeomorphism between Banach spaces (a {\it regular map} in the terminology of \cite{John}) which satisfies
$$
\int_0^{\infty} \inf_{\Vert x \Vert \leq t}  D^{-}_x f \, {\rm d}t = \infty,
$$
then $f$ is in fact a global homeomorphism from $X$ onto $Y$ (analogous results for maps between metric spaces have been obtained in  \cite[Theorem 4.6]{GutuJaramillo} and   \cite[Corollary 7]{GaGuJa}). Furthermore, it was proved in \cite[Theorem IIA]{John} (see also \cite[Theorem 6]{GaGuJa}) that in this case, for each $x_0\in X$ and each $\delta> 0$, we have
\begin{equation}
B\bigl(f(x_0),\, \rho \bigr)\subset f\bigl(\B(x_0,\, \delta)\bigr), \quad \text{where} \quad
\rho = \int_0^{\delta} \inf_{\Vert x - x_0 \Vert \leq t} D_x^- f \, {\rm d}t.\nonumber
\end{equation}

\begin{thm}\label{main-H}
Let $f:X\to Y$ be a continuous map between Banach spaces and let $Jf$ be a pseudo-Jacobian mapping for $f$ satisfying the chain rule condition. Suppose that $Jf$ is regular at every $x\in X$ and satisfies the Hadamard integral condition:
\begin{equation}
\int_0^{\infty} \inf_{\Vert x \Vert \leq t} \alpha_{Jf}(x) \, {\rm d}t =  \infty. \label{int}
\end{equation}
Then,  $f$ is a global homeomorphism from $X$ onto $Y$, whose inverse is Lipschitz on each bounded subset of $Y$. Furthermore, for each $x_0\in X$ and each $\delta> 0$, we have
\begin{equation}
B\bigl(f(x_0),\, \rho \bigr)\subset f\bigl(\B(x_0,\, \delta)\bigr), \quad \text{where} \quad \rho = \int_0^{\delta} \inf_{\Vert x - x_0 \Vert \leq t} \alpha_{Jf}(x) \, {\rm d}t.\label{balls-H}
\end{equation}
\end{thm}
\begin{proof}
According to Theorem \ref{local-inversion} we have that $f$ is a local homeomorphism, and moreover, given any $x\in X$ and any $0 < \alpha < \alpha_{Jf}(x)$, there exists $r> 0$ such that
$$
\Vert f(z_1) - f(z_2) \Vert \geq \alpha\Vert z_1-z_2 \Vert, \quad \text{for all} \quad z_1, z_2\in B(x, r).
$$
In particular, $\Vert f(z) - f(x) \Vert \geq \alpha\Vert z-x \Vert,$ for all $z\in B(x, r)$, and so, $D^{-}_x f \geq \alpha$. Therefore, $$D^{-}_x f \geq \alpha_{Jf}(x), \quad \text{for every} \quad  x\in X.$$ Bearing in mind \eqref{int} we get
$$\int_0^{\infty} \inf_{\Vert x \Vert \leq t}  D^{-}_x f \, {\rm d}t=  \infty.$$ Thus, from
the above mentioned results of F. John  it follows that $f$ is a global homeomorphism from $X$ onto $Y$, and that the inclusion \eqref{balls-H} holds.

It remains to show that the map $f^{-1}:Y \to X$ is Lipschitz on bounded subsets of $Y$. We have already proved that, for each $x\in X$ and each $0< \alpha< \alpha_{Jf}(x)$,  $f^{-1}$ is locally $\alpha^{-1}$-Lipschitz in a neighborhood of $f(x)$. Fix an open ball $B\bigl(f(0), R \bigr)$ in $Y$, and by using condition \eqref{int}, find a number
$r> 0$ such that $\int_0^r \alpha_{Jf}(x) {\rm d}t \geq R$. Taking into account \eqref{balls-H} we obtain
$$B\bigl(f(0); R\bigr)\subset f\bigl(B(0, r)\bigr).$$ Now, set $V_r:= f\left(B(0, r)\right)$ and
$$
\alpha(r) :=\inf \{\alpha_{Jf}(x) \, : \, \Vert x \Vert \leq r \}.
$$
Again from condition \eqref{int}  we get $\alpha(r)>0$. Fix  $0<\alpha <\alpha(r)$. Then, $f$ is an homeomorphism from $B(0, r)$ onto $V_r$ and $f^{-1}$ is locally $\alpha^{-1}$-Lipschitz on $V_r$. Now, from a standard technique, using the convexity of the ball $B\bigl(f(0), R \bigr)$ we deduce that $f^{-1}$ is, in fact, globally $\alpha^{-1}$-Lipschitz on $B\bigl(f(0), R \bigr)$. We give some details for completeness. Choose $u, v\in  B\bigl(f(0), R \bigr)$. Each point $z$ in the segment $[u, v]$ has an open neighborhood $W^z$ such that $f^{-1}$ is locally $\alpha^{-1}$-Lipschitz on $W^z$. This provides an open covering  of $[u, v]$, and by compactness we obtain a Lebesgue number $\delta>0$ such that every subset of $[u, v]$ with diameter less than $\delta$ is contained in some $W^z$. Now, take a number $n \in \mathbb N$ with $n> R\delta^{-1}$, and define $z_j= \frac{n-j}{n}u +\frac{j}{n}v$ for $j=0, 1, \cdots , n$. Then,
$$
\left\Vert f^{-1}(u) - f^{-1}(v)\right\Vert \leq \sum_{j=1}^n \left\Vert f^{-1}(z_{j-1}) - f^{-1}(z_j)\right\Vert \leq \alpha^{-1} \sum_{j=1}^n \Vert z_{j-1} - z_j\Vert = \alpha^{-1} \Vert u-v\Vert,
$$
as we wanted.
\end{proof}

As a consequence of the previous theorem and its proof, we obtain the following result.

\begin{cor}\label{global-1}
Let $f:X\to Y$ be a continuous map between Banach spaces and let $Jf$ be a pseudo-Jacobian mapping for $f$ satisfying the chain rule condition. Suppose that $Jf$ is regular at every $x\in X$, and
\begin{equation}
\alpha_X:= \inf_{x\in X} \alpha_{Jf}(x)> 0. \nonumber
\end{equation}
Then,  $f$ is a global homeomorphism from $X$ onto $Y$, and the map $f^{-1}$ is $\alpha_X^{-1}$-Lipschitz on $Y$.
\end{cor}
\begin{proof}
Condition \eqref{int} is clearly satisfied. Thus, $f: X \to Y$ is a global homeomorphism by Theorem \ref{main-H}. Now fix $0<\alpha < \alpha_X$. Using the argument employed in the proof of that theorem  we also have that for every $x\in X$   the inverse map $f^{-1}: Y \to X$ is locally $\alpha^{-1}$-Lipschitz in a neighborhood of $f(x)$. Then,  as in the same  proof, we can deduce that $f^{-1}$ is globally $\alpha^{-1}$-Lipschitz on $Y$. This gives that $f^{-1}$ is in fact globally $\alpha_X^{-1}$-Lipschitz on $Y$.
\end{proof}

In the case of reflexive spaces, we obtain the following result.

\begin{cor}\label{reflex}
Let $f:X\to Y$ be a locally Lipschitz map between reflexive Banach spaces and let $Jf$ be a locally bounded pseudo-Jacobian for $f$, such that $Jf(x)$ is convex and {\tiny \emph{WOT}}-compact for every $x\in X$. Suppose that $Jf$ is regular at every $x\in X$ and satisfies the Hadamard integral condition:
\begin{equation}
\int_0^{\infty} \inf_{\Vert x \Vert \leq t} \alpha_{Jf}(x) \, {\rm d}t =  \infty.
\end{equation}
Then,  $f$ is a global homeomorphism from $X$ onto $Y$, whose inverse is Lipschitz on each bounded subset of $Y$.
\end{cor}

\begin{proof}
Being reflexive, the space $Y$ admits an equivalent Fr\'echet-smooth norm (c.f. \cite[Corollary 4]{Troyanski}). By Corollary \ref{crc-ex} (i), the pseudo-Jacobian $Jf$ satisfies the chain rule condition for this renorming. Since the regularity of $Jf$ and the Hadamard integral condition are stable under equivalent renormings of the spaces, the conclusion follows from Theorem \ref{main-H}.
\end{proof}

\

As an application of the former corollary, we obtain the following infinite-dimensional version of the global inversion theorem by Pourciau in \cite{Pourciau1988}. Given a locally Lipschitz map $f:X\to Y$  between reflexive Banach spaces, we consider the P\'ales-Zeidan generalized Jacobian $\partial f$ defined in \cite{PZ-07} (see Example \ref{Thibault}) and we denote, for each $x\in X$,
$$
/\hspace{-0.9mm}/ \partial f(x) /\hspace{-0.9mm}/:= \inf \{/\hspace{-0.9mm}/T/\hspace{-0.9mm}/ \, : \, T\in \partial f(x) \}.
$$

\begin{cor}\label{pour}
Let $f:X\to Y$ be a locally Lipschitz map between reflexive Banach spaces, and let $\partial f$ be the P\'ales-Zeidan generalized Jacobian of $f$. Suppose that for each $x\in X$, every $T\in \partial f(x)$ is an isomorphism,
the pseudo-Jacobian $\partial f$ is upper semicontinuous on $U$,  and the following Hadamard integral condition holds:
\begin{equation}
\int_0^{\infty} \inf_{\Vert x \Vert \leq t} /\hspace{-0.9mm}/ \partial f(x) /\hspace{-0.9mm}/ \, {\rm d}t =  \infty.
\end{equation}
Then,  $f$ is a global homeomorphism from $X$ onto $Y$, whose inverse is Lipschitz on each bounded subset of  $Y$.
\end{cor}
\begin{proof}
As we have remarked in Example \ref{Loc-WOT}, the P\'ales-Zeidan generalized Jacobian $\partial f$ defined in \cite{PZ-07} is a locally bounded pseudo-Jacobian for $f$, satisfying that $\partial f(x)$ is convex and {\tiny \emph{WOT}}-compact for every $x\in X$.
Now, by Proposition \ref{usc-2} we have
$$
\alpha_{\partial f}(x) = \inf\{/\hspace{-0.9mm}/T/\hspace{-0.9mm}/: T \in \partial f(x))\}.
$$
Furthermore, by Corollary \ref{crc-ex} (i), $\partial f$ satisfies the chain rule condition after a Fr\'echet-smooth renorming of $Y$, and Corollary \ref{reflex} applies.
\end{proof}

Now, we give some applications of the previous results to the invertibility of small Lipschitz perturbations of differentiable maps between Banach spaces.

\begin{cor}\label{sum}
Let $X, Y$ be reflexive spaces endowed with a Fr\'echet-smooth norm. Consider a map $f:X \to Y$ of the form $f=g+h$, where $g, h: X \to Y$ satisfy:
\begin{enumerate}
\item[(i)] $g$ is $C^1$-smooth on all of X, and $g'(x)$ is an isomorphism for each $x\in X$.
\item[(ii)] $h$ is locally Lipschitz, and ${\rm Lip}\, h(x) < \Vert g'(x)^{-1} \Vert^{-1}$  for each $x\in X$.
\item[(iii)] The following integral condition holds:
$$\int_0^{\infty} \inf_{\Vert x \Vert \leq t} \left( \Vert g'(x)^{-1} \Vert^{-1} - {\rm Lip}\, h(x)  \right) \, {\rm d}t =  \infty.$$
\end{enumerate}
Then, $f$ is a global homeomorphism from $X$ onto $Y$, whose inverse is  Lipschitz on bounded subsets of $Y$.
\end{cor}

\begin{proof}
Consider the set-valued map $Jf :X \to 2^{{\mathcal L}(X, Y)}$ defined, for each $x\in X$, by
$$
Jf(x):=  g'(x) + {\rm Lip}\, h (x) \cdot \overline{B}_ {{\mathcal L}(X, Y)}.
$$
Taking into account Examples \ref{wG} and \ref{Loc-WOT} and the local Lipschitzness of $g$, it follows that $Jf$ is a locally bounded pseudo-Jacobian mapping for $f$,  and  $Jf(x)$ is convex and {\tiny \emph{WOT}}-compact for every $x\in X$. Thus, thanks to Corollary \ref{crc-ex} (i) we deduce that $Jf$ satisfies the chain rule condition. Now, we shall show that $Jf$ is regular. Because of Example \ref{Loc-Lip}, $Jf$ is upper semicontinuous on $X$. Thus, it suffices to see that $Jf$ satisfies conditions $(i)$ and $(ii)$ from Proposition  \ref{usc-2} at every $x\in X$.   For each $R\in {\mathcal L}(X, Y)$ with $\Vert R \Vert \leq {\rm Lip}\, h (x)$  we have $\Vert R \circ g'(x)^{-1}\Vert <1$. In particular, the operator ${\rm Id}_Y + R \circ g'(x)^{-1}$ is an isomorphism on $Y$. In this way we obtain that $g'(x) + R$ is an isomorphism. On the other hand,
$$\begin{aligned}
\alpha_{Jf}(x) &=
\inf\{/\hspace{-0.9mm}/g'(x)+R/\hspace{-0.9mm}/ \, : \,R \in {\mathcal L}(X, Y), \, \Vert R \Vert \leq {\rm Lip}\, h (x)\}
\\
&\geq /\hspace{-0.9mm}/g'(x)/\hspace{-0.9mm}/ - {\rm Lip}\, h (x) = \Vert g'(x)^{-1} \Vert^{-1} - {\rm Lip}\, h(x)>0.
\end{aligned}$$
Therefore, the result follows from Corollary \ref{reflex}.
\end{proof}

\

Our next example gives an similar application of Theorem  \ref{main-H} in the non-reflexive case.

\begin{ex}\label{L1}
Let $(\Omega, \mu)$ be a finite measure space, and consider a map  $f:L^1 (\mu) \to L^1 (\mu)$ of the form $f=g+h$, where $f,g:L^1(\mu)\to L^1(\mu)$ satisfy the following properties:
\begin{enumerate}
\item[(i)] $g$ is $C^1$-smooth on all of $L^1(\mu)$, and $g'(x)$ is an isomorphism for each $x\in L^1 (\mu)$.
\item[(ii)] There exist a number $1<p < \infty$ and a map $\hat{h}:L^1(\mu)\to L^p(\mu)$ such that $\hat{h}$ is locally Lipschitz, being ${\rm Lip}\, \hat{h}(x)< \|g'(x)^{-1}\|^{-1}$ for all $x\in L^1(\mu)$, and so that
$h = \hat{h}\circ i$, where $i:L^p(\mu)\hookrightarrow L^1 (\mu)$ denotes the canonical inclusion.
\item[(iii)] The following integral condition holds:
$$\int_0^{\infty} \inf_{\Vert x \Vert \leq t} \left( \Vert g'(x)^{-1} \Vert^{-1} - 2 \, {\rm Lip}\, h(x)  \right) \, {\rm d}t =  \infty.$$
\end{enumerate}
Then, $f$ is a global homeomorphism of $L^1(\mu)$, whose inverse is Lipschitz on bounded subsets of $L^1(\mu)$.
\end{ex}

\begin{proof}
Let us denote the usual norm of $L^1(\mu)$ by
$
\Vert x \Vert := \int_{\Omega} \vert x (\tau) \vert \, {\rm d}\mu(\tau),\,\, x\in L^1(\mu).
$
According to \cite[Theorem 4.2]{Bo-Fi}, there exists an equivalent norm $\Vert \cdot \Vert_0$ on $L^1(\mu)$ which is weak-Hadamard smooth and satisfies
$$
\Vert x \Vert_0 \leq \Vert x \Vert \leq \sqrt{2}\Vert x \Vert_0.
$$
We will consider $L^1(\mu)$ endowed with the norm $\Vert \cdot \Vert_0$. Note that, for this new norm, we have
$$
\frac{1}{\sqrt{2}} \, \Vert g'(x)^{-1} \Vert^{-1} \leq \Vert g'(x)^{-1} \Vert_0^{-1} \leq \sqrt{2} \,\Vert g'(x)^{-1} \Vert^{-1}
$$
and
$$
\frac{1}{\sqrt{2}} \, {\rm Lip}\, h(x) \leq {\rm Lip}_{\Vert \cdot \Vert_0} h(x) \leq \sqrt{2} \, {\rm Lip}\, h(x).
$$
Therefore,
$$
\Vert g'(x)^{-1} \Vert_0^{-1} -  {\rm Lip}_{\Vert \cdot \Vert_0} h(x) \geq \frac{1}{\sqrt{2}} \left( \Vert g'(x)^{-1} \Vert^{-1} - 2 \, {\rm Lip}\, h(x)  \right)
$$
for every $x\in L^1(\mu)$.

Denote by $\overline{B}_ {{\mathcal L}(L^1, L^p)}$ the unit ball of the space ${\mathcal L}(L^1 (\mu), L^p(\mu))$ when $L^1(\mu)$ is endowed with the norm $\Vert \cdot \Vert_0$ and $L^p(\mu)$ with its usual norm. By composing with the natural inclusion $L^p(\mu)\hookrightarrow L^1 (\mu)$ we consider $\overline{B}_ {{\mathcal L}(L^1, L^p)}$ as a {\tiny \emph{WOT}}-compact subset of ${\mathcal L}(L^1 (\mu), L^1(\mu))$. Now define the set-valued map $Jf :X \to 2^{{\mathcal L}(L^1 (\mu), L^1(\mu))}$ by
$$
Jf(x):=  g'(x) + {\rm Lip}_{\Vert \cdot \Vert_0} h (x) \cdot \overline{B}_ {{\mathcal L}(L^1, L^p)}.
$$
Taking into account Examples \ref{wG} and \ref{Loc-WOT}, we have that  $Jf$ is a pseudo-Jacobian mapping for $f$ such that $Jf(x)$ is convex and {\tiny \emph{WOT}}-compact for every $x\in L^1(\mu)$. From Example \ref{Loc-Lip}, we also have that $Jf$ is upper semicontinuous on $L^1(\mu)$. Moreover, given $x, v\in L^1(\mu)$, there exist constants $r, C> 0$ such that
$$
Jf ([x, x+rv]) \subset g'([x, x+rv]) + C \cdot \overline{B}_ {{\mathcal L}(L^1, L^p)},
$$
and bearing in mind that the map $g$ is $C^1$-smooth, it follows that $g'([x, x+rv]) + C \cdot \overline{B}_ {{\mathcal L}(L^1, L^p)}$ is a {\tiny \emph{WOT}}-compact subset of ${\mathcal L}(L^1 (\mu), L^1(\mu))$. In this way, using Corollary \ref{crc-ex} (iii), we obtain that $Jf$ satisfies the chain rule condition. Finally, the regularity of $Jf$ can be checked as in the previous result. Thus, the conclusion follows from Theorem \ref{main-H}.

\end{proof}

\

In the following example we consider the global invertibility of different types of 1-Lipschitz perturbations of the identity in a Hilbert space.

\begin{ex}
\rm{
Consider the separable Hilbert space $\ell_2$ and define the map $f:\ell_2 \to \ell_2$ by the formula
$$
f(x_1, x_2, x_3, \dots) = (x_1+ \theta (\vert x_2 \vert), x_2+ \theta(\vert x_3\vert), x_3+\theta(\vert x_4\vert), \dots),
$$
where $\theta:[0, +\infty) \to \mathbb R$ is given below. Then we have:
\begin{enumerate}
\item[(a)] Suppose that $\theta (t)=ct$, with $\vert c \vert <1$. Then, $f$ is a global homeomorphism with Lipschitz inverse.
\item[(b)] Suppose that $\theta (t)=t$. Then, $f$ is not bijective.
\item[(c)] Suppose that $\theta (t)=t-\log(1+t)$. Then, $f$ is a global homeomorphism, whose inverse is Lipschitz on bounded sets.
\end{enumerate}
\begin{proof}
We have that $f=g+h$, where $g:\ell_2 \to \ell_2$ is the identity map and $h:\ell_2 \to \ell_2$ is defined by
$$
h(x_1, x_2, x_3, \dots) = (\theta (\vert x_2 \vert), \theta(\vert x_3\vert), \theta(\vert x_4\vert), \dots).
$$

Consider first case (a). It is clear that in this case $h$ is $\vert c \vert$-Lipschitz on $\ell_2$.  Then for each $x\in \ell_2$ we have
$$
\Vert g'(x)^{-1} \Vert^{-1} - {\rm Lip}\, h(x) \geq 1- \vert c \vert >0.
$$
Thus, the conclusion follows from Corollary \ref{global-1} and the proof of Corollary \ref{sum}.

\

Next, consider case (b). Here we have that $f$ is not onto. Indeed, consider the vector $y=\left( \frac{-1}{n}\right)_{n \in \mathbb N} \in \ell_2$ and suppose that there exists some $x=(x_n)_{n \in \mathbb N} \in \ell_2$ with $f(x)=y$. Then, for each $n\in \mathbb N$ we have
$$
x_n = -\frac{1}{n}- \vert x_{n+1}\vert \leq -\frac{1}{n}.
$$
Therefore, for each $n\in \mathbb N$ we obtain
$$
x_1=-1-\vert x_2\vert\leq -1-\frac{1}{2}-\vert x_3\vert\leq \cdots \leq  -1-\frac{1}{2}-\cdots - \frac{1}{n}-\vert x_{n+1}\vert,
$$
which is not possible.

\

Finally, consider case (c). Given $t>0$, from elementary calculus we obtain
$$
\vert \theta (\vert r \vert) -\theta (\vert s\vert)\vert \leq \theta'(t)\, \vert r-s\vert = \frac{t}{1+t}\, \vert r-s\vert, \quad \text{for all} \quad r,s\in \mathbb R \quad \text{such that} \quad |r|, |s|\leq t.
$$
From here we deduce that the map $h$ is $\frac{t}{1+t}$-Lipschitz on the closed ball $\overline{B}(0, t)$ of $\ell_2$. As a consequence,
$$
\sup_{\Vert x \Vert \leq t} {\rm Lip}\,h(x) \leq \frac{t}{1+t}.
$$
Therefore,
$$\begin{aligned}
\int_0^{\infty} \inf_{\Vert x \Vert \leq t} \left( \Vert g'(x)^{-1} \Vert^{-1} - {\rm Lip}\, h(x)  \right) \, {\rm d}t
&= \int_0^{\infty} \inf_{\Vert x \Vert \leq t} \left( 1 - {\rm Lip}\, h(x)  \right) \, {\rm d}t\geq
\int_0^{\infty} \left( 1 - \frac{t}{1+t} \right) \, {\rm d}t \\ &= \int_0^{\infty} \frac{1}{1+t} \, {\rm d}t= \infty,
\end{aligned}$$
and the result follows from Corollary \ref{sum}.
\end{proof}
}
\end{ex}

We end this paper with a characterization of homeomorphisms between Banach spaces in terms of pseudo-Jacobians, which relies on our local invertibility result and Plastock's limiting condition for global invertibility of local homeomorphisms in \cite[Theorem 2.1]{Plastock}.

\

\begin{thm}\label{main2}
Let $f:X\to Y$ be a continuous map and let $Jf$ be a pseudo-Jacobian mapping for $f$ satisfying the chain rule condition. Suppose that, for every $x\in X$, we have that  $Jf$ is regular at $x$. Then the following conditions are equivalent:
\begin{enumerate}
\item[(a)] $f:X \to Y$ is a global homeomorphism;
\item[(b)] for each compact subset $K\subset Y$,
$$
\inf \{\alpha_{Jf}(x) \, : \, x\in f^{-1}(K)\} > 0;
$$
\item[(c)] $f:X \to Y$ is a global homeomorphism with a locally Lipschitz inverse.
\end{enumerate}
\end{thm}

\begin{proof}
Suppose that $(a)$ holds, and let $K$ be a compact subset of $Y$. Then, $f^{-1}(K)$ is  compact in $X$. From the very definition, it is easily seen that the function given by $x\mapsto \alpha_{Jf}(x)$ is lower semi-continuous on $X$. Therefore $\alpha_{Jf}(x)$ attains its minimum on  $f^{-1}(K)$, and taking into account that $\alpha_{Jf}(x)> 0$ for all $x\in X$, we obtain
$$
\inf\{\alpha_{Jf}(x) \, : \, x\in f^{-1}(K)\} > 0,
$$
so that condition $(b)$ is satisfied.

\

Now, suppose that assertion $(b)$ holds. From Theorem \ref{local-inversion} we have that $f$ is a local homeomorphism, and furthermore (using the same argument as in the proof of Theorem \ref{main-H}), for each $x\in X$ and each  $0< \alpha< \alpha_{Jf}(x)$ there exists $r> 0$ such that
\begin{equation}
\Vert f(z) - f(x) \Vert \geq \alpha\Vert z-x \Vert, \quad \text{for all} \quad z\in B(x, r).\label{lipschitz-local}
\end{equation}
Hence, $D^{-}_x f \geq \alpha$. In this way we see that $D^{-}_x f \geq \alpha_{Jf}(x)$ for every $x\in X$. Bearing in mind $(b)$ we have
\begin{equation}
\inf_{x\in f^{-1}(K)} D^{-}_{x} f > 0, \quad \text{for every compact set} \quad K\subset Y.\label{compact-scalarderivative}
\end{equation}
In order to see that $f$ is a global homeomorphism, according to \cite[Theorem 1.2]{Plastock} it suffices to show that $f$ satisfies the the following \emph{limiting condition}:

\medskip

{\rm (L)} \emph{For every line segment $p:[0, 1] \to Y$,  for every $0<b\leq 1$ and for every continuous path $q:[0, b)\to X$ satisfying that $f(q(t))=p(t)$ for every $t\in [0, b)$, there exists a sequence $(t_j)_j$ in $[0, b)$ convergent to $b$ and such that the sequence $\left(q(t_j)\right)_j$ converges in $X$.}

\medskip

This can be done arguing as in the proof of \cite[Theorem 4.3]{JaMaSa1}. We include the details for the sake of completeness. Consider a line segment $p:[0, 1] \to Y$ given by $p(t)=(1-t)y_0+ty_1$ for some $y_0, y_1\in Y$, pick $0<b\leq 1$, and let $q:[0, b)\to X$ be a continuous path satisfying $f(q(t))=p(t)$ for every $t\in [0, b)$. Let $K:=p([0,1])$. Then, $K$ is a compact subset of $Y$ and $q([0, b))\subset f^{-1}(K)$. So, by \eqref{compact-scalarderivative}, there is $\alpha >0$ such that $\alpha_{Jf}(x) > \alpha>0,$ for all  $x\in q([0, b)).$
Let $(t_j)_j$ be any sequence in $[0,b)$ such that $t_j\to b$. According to the Mean Value Inequality in \cite[Proposition 3.9]{GutuJaramillo}  we have
$$\|q(t_i)-q(t_j)\|\leq \alpha^{-1}\|p(t_i)-p(t_j)\|= \alpha^{-1}|t_i-t_j|\|y_0-y_1\|, \quad \text{whenever} \quad i,j\in \mathbb N.$$
Thus, the sequence $(q(t_j))_j$ is Cauchy, and therefore convergent in $X$. Thus, condition $(L)$ is fulfilled, and $f$ is an homeomorphism from $X$ onto $Y$. Moreover, thanks to \eqref{lipschitz-local} we have that the inverse map $f^{-1}$ is locally Lipschitz on $Y$.

\end{proof}

\end{document}